\documentclass[lettersize,journal]{IEEEtran}
\usepackage{amsmath}
 \usepackage{amsthm}
\usepackage{amsfonts}
\usepackage{algorithmic}
\usepackage{algorithm}
\usepackage{array}
\usepackage[caption=false,font=normalsize,labelfont=sf,textfont=sf]{subfig}
\usepackage{textcomp}
\usepackage{booktabs}
\usepackage{stfloats}
\usepackage{url}
\usepackage{verbatim}
\usepackage{multirow}
\usepackage{multicol}
\usepackage{graphicx}
\usepackage{cite}
\usepackage{xcolor}
\usepackage{hyperref}

\usepackage{float}
\usepackage{booktabs} 
\usepackage[export]{adjustbox}

\newcommand{\xkik}{{x}_{k,i_{k}}} 
\newcommand{\hxkik}{\hat{x}_{k,i_{k}}}

\DeclareMathOperator*{\argmin}{arg\,min}

\newtheorem{remark}{Remark}
\newtheorem{lem}{Lemma}

\newtheorem{thm}{Theorem}
\newtheorem{assum}{Assumption}

\title{Regularized coordinate minimization for nonconvex composite optimization with application to quantized image compression}

\author{Daniela Lupu$^{1,2}$, George  T. Samoila$^{1}$,  Adina M. Florea$^{3}$, Ion Necoara$^{1,4}$
\thanks{*The research leading to these results has received funding from:  project Romanian Hub for Artificial Intelligence - HRIA, Smart Growth, Digitization and Financial Instruments Program, 2021-2027, SMIS no. 2021- 351416.}
\thanks{$^{1}$Automatic Control and Systems Engineering Department, Politehnica  Bucharest, 060042 Bucharest, Romania. Corresponding author email: {\tt\small  ion.necoara@upb.ro.}}%
	
\thanks{$^{2}$ Precis Research Institute, Politehnica  Bucharest, 060042 Bucharest. }

\thanks{$^{3}$  Department of Computer Science,  Politehnica  Bucharest, 060042 Bucharest, Romania. }
    
\thanks{$^{4}$Gheorghe Mihoc-Caius Iacob Institute of Mathematical Statistics and Applied Mathematics of the Romanian Academy, 050711 Bucharest, Romania.}
}

\begin{document}

\maketitle

\begin{abstract} This paper presents  a regularized cyclic coordinate minimization method for solving  nonconvex composite  optimization problems having  the  objective function formed as the sum of two  terms, one is twice continuously differentiable  and the second  term is simple and  separable.  We analyze the  convergence  behaviour of  our  coordinate minimization  method, in particular we provide convergence rates to a first-order optimality criterion and objective residual  depending on the assumptions on the problem. Then, we show that our algorithmic framework can be efficiently applied for solving quantized matrix factorization problems that arise in e.g.,  lossy image compression. More specifically, on KODAK and  CLIC 2024 datasets,  our method notably outperforms JPEG at low bit rates, achieving savings in bits per pixel without excessive degradation and remains comparable at higher bit rates. Moreover, we evaluate both  the original and  quantized compressed images in a classification task using 3 well-known convolutional networks AlexNet, ResNet50 and MobileNetV2 based on floating point and integer arithmetic representations. Remarkably, when comparing the two numerical representations of the convolutional networks, there is no substantially accuracy degradation on  ImageNet dataset. The numerical results on image compression and classification using real data show the flexibility and  efficiency of our  algorithm compared to well-established methods from the literature. 
\end{abstract}


\section{Introduction}
\noindent In many applications ranging from signal processing  (e.g., dimensionality reduction, data compression) \cite{Hou:15} to distributed control  \cite{NecCli:13}, one needs to solve  large-scale nonconvex composite optimization problems. However, for large-scale nonconvex  optimization,  the traditional methods based on full gradient/prox  and on Hessian  computations are prohibitive. In this case, a reasonable approach to solve such problems  is to use  (block) coordinate descent methods as each iteration is cheap since it involves optimizing a small-dimensional subproblem, see~\cite{Ber:99}. Comprehensive  surveys on coordinate descent algorithms can be found in e.g., \cite{Nec-cd17,wright2015coordinate}.

\medskip 

\noindent The variants of  (block) coordinate descent algorithms differ from each other in the way  we define the local approximation function over which we optimize  and the criterion of choosing at each iteration the subspace over which we minimize this local approximation function.  For updating one (block) variable, while keeping the other variables fixed,  two basic choices for the local approximation function  are usually considered: (i) exact approximation function, leading to \textit{coordinate minimization methods} \cite{Ber:99, NecCho:24},  and (ii)  quadratic approximation function, leading to \textit{coordinate (proximal) gradient}  \textit{descent   methods} \cite{BecTet:13,NecCli:16,NecCho:24,TseYun:09}.  Furthermore, three classical criteria for choosing the coordinate search  used often in these algorithms are the greedy, the cyclic and the random coordinate search, which differ significantly in the computational effort required to select the next coordinate. For cyclic coordinate search rates of convergence have been given only recently \cite{BecTet:13,Teb:14}. Convergence rates for coordinate descent methods based on the Gauss-Southwell rule  were given e.g.  in \cite{TseYun:09}. Finally, a third  approach is based on stochastic coordinate descent, where the coordinate search is random. Recent complexity results on stochastic coordinate descent methods were obtained  e.g., in \cite{FerRic:15,NecCli:16} in the convex settings and then extended to nonconvex problems as well, see e.g.,  \cite{NecCho:24}. 

\medskip 

\noindent \textit{Contributions.} This paper deals with nonconvex composite  optimization problems having  the  objective function formed as the sum of two  terms, one is twice continuously differentiable  and the second  term is simple and (block) separable, both of which may be nonconvex.  Under these settings we design a new  \textit{regularized cyclic  coordinate minimization} method which takes into account the composite form of the objective function. This algorithm achieves scalability by minimizing at each iteration a local model of the whole nonseparable  objective function along  a  subspace with a user-determined dimension.   We provide convergence  analysis for our  coordinate minimization  method, in particular we prove worst-case bounds in terms of the number of iterations for achieving approximate first-order optimality and objective residual. Our convergence rates range between sublinear to (super)linear depending on whether our objective function is general or satisfies  a  Kurdyka-Lojasiewicz (KL) property.  Note that our algorithm is different from e.g., the proximal alternating  linearized  algorithm  from \cite{Teb:14}, which minimizes a quadratic approximation of the composite objective at each iteration based on the linearization of  the differentiable term along a subspace and an estimate of the coordinatewise Lipschitz constant on that subspace.  On the other hand, our method uses a local model of the whole objective along a subspace instead of a quadratic approximation and does not need to know any Lipschitz constant estimates. 

\medskip 

\noindent Then, we show that quantized matrix factorization (QMF) problems fit into our algorithmic framework, leading to a method with simple (closed-form) iterations and enjoying fast convergence.  QMF is an extension of the traditional  matrix factorization \cite{gillis2020nonnegative}  that explicitly incorporates quantization constraints  into the optimization problem (i.e., the  elements of the factor matrices are constrained to a set of bounded integer values), see \cite{Fre:14,fedotov2025qmf, qmf}.   The quantization step allows for a more effective low-rank approximation for data compression tasks, such as image compression, being useful when dealing with storing and transmission tasks, particularly when there are hardware limitations. For example, in lossy image compression, quantization  reduces storage size by eliminating redundant information and details that are less visually detected by humans. Our extensive numerical results using  Kodak \cite{kodak1993} and  CLIC 2024 \cite{clic2024}  datasets, confirm the   efficiency of our coordinate minimization  algorithm on solving QMF problems  for lossy image compression, producing bits per pixel (bpp) savings without excessive degradation. Remarkably, our method outperforms JPEG and quantized compression algorithm in \cite{qmf} in several metrics (PSNR and SSIM) at low bit rates ($<0.3$ bpp)  and remains comparable at higher bit~rates.

\medskip 

\noindent We also investigate the impact of the  quantized image compression in  learning tasks such as classification. Driven by the need to improve deep neural networks (DNNs) speed and cost-effectiveness, especially when deployed on hardware with limited capabilities, there is growing demand for frameworks that support low-precision arithmetic. Hence, we consider both, floating point and integer arithmetic representations for some widely used convolutional neural network architectures such as AlexNet \cite{alexnet}, ResNet50 \cite{resnet} and MobileNetV2 \cite{mobilenetv2}. For inference we use the original ImageNet  validation  dataset (50.000 images) \cite{deng2009imagenet},  but also the quantized compressed ones and evaluate the models accuracy on  the original DNNs and the  post training quantized version of these networks \cite{gholami2022survey, wang2022niti}.  Comparing the two numerical representations of the convolutional networks AlexNet, ResNet50  and MobileNetV2, one can notice that there is no substantially accuracy degradations on (quantized and compressed) ImageNet dataset.

\medskip

\noindent Hence, our key contributions can be summarized as follows:
\begin{itemize}
    \item We propose a new scalable regularized cyclic coordinate minimization method for nonconvex composite optimization problems, where the objective is the sum of a twice continuously differentiable term and a simple block-separable term, both potentially nonconvex.
    \item We establish convergence guarantees for the algorithm, including iteration-complexity bounds for reaching approximate first-order stationary points and objective residual.
    \item Then, we show that our algorithm can efficiently  solve the quantized matrix factorization (QMF) problem that can be used, e.g.,  in lossy image compression tasks. We show that our method when applied to QMF  has very simple iterations (matrix-vector operations) and fast convergence. The numerical results on image compression and classification using real data show the flexibility and  efficiency of our optimization  algorithm.      
\end{itemize}

\medskip

\noindent \textit{Notations.}
We denote with $\mathbb{Z}$ and  $\mathbb{R}$ the set of integer and real numbers and $\mathbb{Z}_{[\alpha, \beta]} = [\alpha, \beta] \cap \mathbb{Z}$. For a given vector $x \in \mathbb{R}^n$, its projection  onto set $\mathcal{X}$ is denoted $[x]_{\mathcal{X}}$. For a scalar $y \in \mathbb{R}$, we denote  $[y]_{\mathbb{Z}_{[\alpha, \beta]}} = \min(\beta, \max(\text{round}(y),\alpha)) = \text{clamp}(\alpha, \beta,(\text{round}(y))$, where $\text{round}(y)$ returns the integer value closest to $y$. Let $E \in \mathbb{R}^{n \times n}$ be  a column permutation of the $n \times n$ identity matrix and  let $E = [E_{1},...,E_{N}]$ be a decomposition of $E$ into $N$ submatrices, with $E_{i} \in \mathbb{R}^{n \times n_{i}}$, where $\sum_{i = 1}^{N}n_{i} = n$. Then, any vector $x \in \mathbb{R}^n$ can be written uniquely as $x = \sum_{i = 1}^{N} E_{i} x^{(i)}$, where $x^{(i)} = E_{i}^{T}x \in \mathbb{R}^{n_{i}}$. For a differentiable function $f :  \mathbb{R}^{n} \to  \mathbb{R}$, $\nabla f(x) \in \mathbb{R}^{n}$ denotes the gradient of $f$ at $x$, $\nabla_{j} f(x) \in \mathbb{R}$ denotes the $j$th component of the gradient, while $\nabla_{(i)} f(x) = E_i^T \nabla f(x) \in \mathbb{R}^{n_{i}}$ denotes the $i$th block component of the gradient. We denote $\nabla^2 f(x)$  the Hessian of $f$. We also denote $\partial f(x)$ the limiting subdifferential (or simply the subdifferential) of $f$ at $x$.

\medskip

\noindent \textit{Content.} In Section 2 we introduce our optimization problem and the main assumptions followed  by the presentation of the new coordinate minimization algorithm and its convergence analysis. In Section 3 we show that  the quantized matrix factorization problem can be recast into our algorithmic framework, while Section 4 provides detailed numerical results for image compression  and its efficiency in a classification task.


\section{Regularized coordinate minimization for nonconvex composite optimization}
\noindent In this section we present our problem of interest and the main assumptions, then we introduce a new coordinate descent method and analyze its convergence properties. 

\subsection{Problem formulation}
\noindent We consider the following nonconvex composite optimization problem (see also \cite{Teb:14, NecCho:24}): 
\begin{equation} 
\label{eq:prob}
F^* = \min_{x \in \mathbb{R}^n} F(x):= f(x) + g(x),
\end{equation}
where  $f$ is a differentiable (possibly nonconvex) function and $g$ is a proper lower semi-continuous function (e.g., the indicator function of a nonconvex set). Such problems arise in many applications ranging e.g., from signal processing   \cite{Hou:15} to distributed control \cite{NecCli:13}.  Throughout the paper the following assumptions are valid for the  composite  problem \eqref{eq:prob}:

\begin{assum} 
\label{ass1}
(i) Function  $f$ is twice continuously differentiable, function $g$ is proper, lower semi-continuous and separable, i.e.:
\[  g(x) = \sum_{i=1}^N g_i(x^{(i)}),    \]
and for any given $\hat y = (y^{(1)},\cdots, y^{(i-1)},  y^{(i+1)},\cdots y^{(N)})$ 
\[ \inf_{x^{(i)} \in \mathbb{R}^{d_i}}  f(x^{(i))}, \hat y) + g_i (x^{(i)}) > - \infty. \]  
(ii) A solution exists for \eqref{eq:prob}  (hence,   $F^* > -\infty$).
\end{assum}
   
\noindent  Optimization problems having the composite structure \eqref{eq:prob}
permit to handle general coupling functions (e.g.,  $f(x^{(1)}, x^{(2)})= \|(x^{(1)})^TA x^{(2)} - y\|^p$, with $A$ linear operator and $p\geq 2$). In particular, this model satisfying Assumption 1 is suitable for applications where for the first term, although possibly smooth, the computations of its gradient or of the coordinatewise Lipschitz constants over a bounded set along  subspaces generated by $E_i$'s are expensive or even impossible; on the other hand the computation of the prox  along a subspace is easy.  Let us fix some notations.  For a given iteration counter $k$, we denote $x_{k,0} = x_{k}$ and then denote by
    \vspace{-0.2cm}
    \begin{align*}
    & \xkik= \left(x_{k+1}^{(1)},\cdots,x_{k+1}^{(i_{k})},x_{k}^{(i_{k}+1)}, \cdots, x_{k}^{(N)}\right),\\
   & \hxkik = \left(x_{k+1}^{(1)},\cdots,x_{k+1}^{(i_{k}-1)},x_{k}^{(i_{k}+1)}, \cdots, x_{k}^{(N)}\right).
    \vspace{-0.2cm}
    \end{align*}


\subsection{Regularized coordinate minimization algorithm}
\noindent In this section, we present a \textit{Regularized Cyclic Coordinate  Minimization}  (RCCM) algorithm for solving problem \eqref{eq:prob},  with $f$  possibly nonseparable and nonconvex, while $g$ is separable and nonconvex. The basic idea of our algorithm consists of updating  the $i$th component of  $x\in\mathbb{R}^{n}$   in a cyclic manner  using  a local model of the whole nonseparable  objective function along  a  subspace in contrast to a quadratic approximation along that subspace as considered in \cite{Teb:14}. 

\begin{algorithm}
\caption{RCCM} \label{alg:RCPG}
		\begin{algorithmic} 	
			 \STATE Given the regularization parameters $M_i\geq 0 \;\;  \forall i=1:N$ and a starting point  $x_{0} \in \text{dom}\,g$
			\FOR{$k \geq 0$}
			\FOR{$i_{k} = 1:N$}
				     \STATE Compute: $x_{k+1}^{(i_{k})} \in$
                     \begin{align}  \label{eq_subpr}
                      \hspace{-0.7cm}\argmin_{x^{(i_k)} \in \mathbb{R}^{n_{i_k}} }  f(x^{(i_k)},  \hat x_{k, i_k}) \!+\!  g_{i_k} (x^{(i_k)})\!+\! \frac{M_{i_k}}{2} \|x^{(i_k)}\! -\! x_k^{(i_k)}\|^2  
                     \end{align}
             \vspace{-0.5cm}        
			\ENDFOR
            \STATE Update: $x_{k+1} = x_{k,N} = \left(x_{k+1}^{(1)},\cdots,x_{k+1}^{(N)}\right)$.
			\ENDFOR
		\end{algorithmic} 
	\end{algorithm}

\noindent Note that the RCCM algorithm  is different from e.g., proximal alternating  linearized minimization (i.e., proximal coordinate gradient) algorithm  from \cite{Teb:14} in both, assumptions and iterations. Indeed, paper \cite{Teb:14} assumes that the smooth term $f$ has coordinatewise Lipschitz gradient, while we just need to assume twice differentiability of $f$. Moreover, \cite{Teb:14} linearizes the smooth term $f$ along a subspace and uses an estimate of the coordinatewise Lipschitz constant on that subspace to produce the next iterate. Our method however  uses a local model of the whole objective along a subspace instead of a quadratic approximation and does not need to know any Lipschitz constant estimates.  It should be observed that Assumption 1 (i) ensures that the set of minima of  subproblem \eqref{eq_subpr} is nonempty and compact, see  \cite{Teb:14} and, consequently, the RCCM  algorithm is well defined.  We also assume in the sequel that the subproblem \eqref{eq_subpr} is easy to solve, e.g., it has a closed-form solution (see the quantization-based matrix factorization application from the next section).


\subsection{Convergence analysis of RCCM algorithm}
\noindent In this section we analyze the convergence behavior of the RCCM algorithm. Let us denote by $M_{\min} = \min (M_{1},\cdots, M_N)$ and $M_{\max} = \max (M_{1},\cdots, M_N)$. Next lemma shows that the objective function of \eqref{eq:prob} \textit{strictly decreases} along the iterations of RCCM algorithm. 

\begin{lem}
		\label{lemma:1}
		If Assumption 1 holds and  the regularization parameters $M_i> 0 \;\;  \forall i=1:N$, then the iterates of the RCCM algorithm  satisfy the following strict descent:
		\begin{align*}
			F(x_{k+1}) \leq F(x_{k}) - \dfrac{M_{\min}}{2} \|x_{k+1} - x_{k}\|^2. 
		\end{align*}
        We also have $\sum_{k=1}^\infty \|x_{k+1} - x_{k}\|^2 < \infty$ and thus $\lim_{k \to \infty} \|x_{k+1} - x_{k}\| =0$.  
\end{lem}	

\begin{proof}
First, one can notice that:
\[  \sum_{i=1}^N f(x_{k,i}) - f(x_{k,i-1}) = f(x_{k+1}) - f(x_k).   \]
Second,  using the optimality  of  $x_{k+1}^{(i_{k})}$ we obtain:
\begin{align*}  
& f(\xkik) + g_{i_k}(x_{k+1}^{(i_k)}) + \frac{M_{i_k}}{2} \|x_{k+1}^{(i_k)} - x_k^{(i_k)}\|^2  \\
& \leq f(x_{k,i_k-1}) + g_{i_k}(x_{k}^{(i_k)}) \quad \forall i_k=1:N. 
\end{align*}
Adding the previous relation  from $i_k=1:N$ and combining with the first relation, we get
\begin{align*}
F(x_{k+1}) \leq  F(x_{k}) -   \sum_{i=1}^N  \frac{M_{i}}{2} \|x_{k+1}^{(i)} - x_k^{(i)}\|^2,  
\end{align*}
which yields the first statement. Further, for any positive integer $K$, summing the previous descent relation from $k=0$ to $k=N-1$ we also get
\[ \dfrac{M_{\min}}{2} \sum_{k=0}^N  \|x_{k+1} - x_{k}\|^2 \leq F(x_0) - F(x_N) \leq F(x_0) - F^*.  \]
Since $M_{\min} >0$, the second statement also follows. 
\end{proof}

\noindent Next lemma establishes a relation between the subdifferential of the objective function and the  difference between consecutive iterates. 
\begin{lem}
		\label{lemma:2}
Let Assumption 1 (i) hold.  Then, we have:  
		\begin{align*}
			\left[\!\nabla_{(i)} \! f(x_{k+1}) \!-\! \nabla_{(i)} \!f(x_{k,i}) -\! M_i \!\left(x_{k+1}^{(i)} \!\!-\! x_k^{(i)}\right)\!\right]_{i=1}^N \!\!\in\! \partial F(x_{k+1}).  
		\end{align*} 
\end{lem}	

\begin{proof}
Using the optimality condition for $x_{k+1}^{(i_{k})}$ and differentiability of $f$, we get:
\[  0 \in  \nabla_{(i_k)} f(x_{k,i_k}) + \partial g_{i_k} ( x_{k+1}^{(i_{k})}) + \! M_{i_k} \!\left(x_{k+1}^{(i_k)} \!- x_k^{(i_k)}\right),  \]
that is
\[  - \left(  \nabla_{(i_k)} f(x_{k,i_k}) + M_{i_k} \!\left(x_{k+1}^{(i_k)} \!- x_k^{(i_k)}\right) \right) \in  \partial g_{i_k} ( x_{k+1}^{(i_{k})}). \] 
Combining this relation with  basic calculus rules for subdifferentiability and  Assumption 1 (i), we finally get:
\begin{align*}
& \partial F(x_{k+1})  = \nabla f(x_{k+1}) +  \partial g(x_{k+1}) \\
& = \left[ \nabla_{(i)} f(x_{k+1}) +  \partial g_i(x_{k+1}^{(i)}) \right ]_{i=1}^N \\
& \ni \left[ \nabla_{(i)} f(x_{k+1}) - \left(  \nabla_{(i)} f(x_{k,i}) + M_{i} \!\left(x_{k+1}^{(i)} \!- x_k^{(i)}\right) \right)  \right ]_{i=1}^N, 
\end{align*}
which proves the statement of the lemma. 
\end{proof}

\noindent Next lemma establishes a relation between the norms of the  subdifferential of the objective function and of the difference between consecutive iterates.  
\begin{lem}
		\label{lemma:3}
Let Assumption 1 (i) hold and assume that the iterates of the RCCM algorithm  are bounded.  Then, there exists $C>0$ such that 
		\begin{align*}
			\min_{s_{k+1} \in \partial F(x_{k+1})} \|s_{k+1}\| \leq C \|x_{k+1} - x_{k}\|.   
		\end{align*} 
\end{lem}	

\begin{proof}  
From Lemma 	\ref{lemma:2} we have:
\begin{align*}
& \min_{s_{k+1} \in \partial F(x_{k+1})} \|s_{k+1}\| \\
& \leq  \left \| \left[\nabla_{(i)} f(x_{k+1}) -\! \nabla_{(i)} f(x_{k,i}) -\! M_i \!\left(x_{k+1}^{(i)} \!- x_k^{(i)}\right)\!\right]_{i=1}^N \right \| \\ 
& \leq \left \| \left[\nabla_{(i)} f(x_{k+1}) -\! \nabla_{(i)} f(x_{k,i}) \right]_{i=1}^N \right \|  \\
& \qquad \qquad + \left \| \left[ \! M_i \!\left(x_{k+1}^{(i)} \!- x_k^{(i)}\right)\!\right]_{i=1}^N \right \| \\
& \leq \left \| \left[\nabla_{(i)} f(x_{k+1}) -\! \nabla_{(i)} f(x_{k,i}) \right]_{i=1}^N \right \|  \!+\!  M_{\max} \! \left \| x_{k+1} \!- x_k \right \|.  
\end{align*}
 Let us now introduce the notation:  
\begin{equation}
	\bar{\nabla}^2 f (z{(1)},\cdots,z{(d)}) = \begin{bmatrix}
		\nabla_{1}^2 f (z{(1)})  \\
		\vdots \\		
		\nabla_{d}^2 f (z{(d)})
	\end{bmatrix}, 
\end{equation}	
\noindent with $\nabla_{j}^2 f (z{(j)})$ being the $j$th row of the hessian of $f$ at the point $z{(j)} \in \mathbb{R}^{n}$. By the mean value theorem, we have that there exist $d$ vectors $z_k{(j)} \in [x_{k,i},x_{k+1}]$, where $i$ is the block index containing the index $j$, such that: 
	\begin{equation*}
		\nabla_{j} f (x_{k+1}) -	\nabla_{j} f(x_{k,i}) \!=\! \nabla_{j}^2 f (z_k{(j)}) (x_{k+1} \!- x_{k,i})   \; \forall  j\!=\!1\!:\!d.
	\end{equation*} 
Since the sequence $(x_k)_{k \geq 0}$ is assumed bounded, there exists a finite positive constant $C_f$ such that $\| \nabla_{j}^2 f (z_k{(j)}) \| \leq C_f$ for all $k \geq 0$ and $j=1:d$.  Hence, from Cauchy-Schwarz inequality we further get: 
\begin{align*}
  & | \nabla_{j} f (x_{k+1}) -	\nabla_{j} f(x_{k,i})  |  \leq C_f \| x_{k+1} \!- x_{k,i} \| \\
  & \leq C_f \|  x_{k+1} \!- x_{k} \|  \quad \forall k \geq 0, j=1:d.  
\end{align*} 
 Finally, we get:
\begin{align*}
& \min_{s_{k+1} \in \partial F(x_{k+1})} \|s_{k+1}\| \\
& \leq \left \| \left[\nabla_{(i)} f(x_{k+1}) -\! \nabla_{(i)} f(x_{k,i}) \right]_{i=1}^N \right \|  \!+\!  M_{\max} \! \left \| x_{k+1} \!- x_k \right \|\\
& \leq (C_f \sqrt{n} + M_{\max}) \left \| x_{k+1} \!- x_k \right \|, 
\end{align*} 
which proves our statement for $C = C_f \sqrt{n} + M_{\max}$.  
\end{proof}

\noindent Note that  RCCM algorithm  does not require the knowledge of  the coordinatewise Lipschitz constants of the differentiable function  $f$, as for example proximal alternating linearized minimization algorithm in  \cite{Teb:14}. Additionally, the constant $C$ from the previous lemma  only appears in the convergence rates, our algorithm does not require explicit knowledge of it. In some applications  the first term, $f$, although differentiable might have expensive  gradient evaluation or  the corresponding  coordinatewise Lipschitz constants  might be difficult to estimate; on the other hand,  if the computation of the proximal step  from subproblem \eqref{eq_subpr} is easy, then  the RCCM algorithm  can be used efficiently (see the next section for such proximal friendly functions). Now, we are ready to derive the main convergence rate for the RCCM algorithm, which matches the usual rates for coordinate descent methods from the literature, see   \cite{NecCho:24, Teb:14}.

\begin{thm}
Let Assumption 1 hold and assume that the iterates of the RCCM algorithm, with  the regularization parameters $M_i> 0 \;  \forall i=1:N$,  are bounded. Then, the following convergence rate in the first-order stationarity  holds:  
\begin{align}
\label{conv1}
\min_{k=0: K}  \min_{s_{k+1} \in \partial F(x_{k+1})} \|s_{k+1}\|^2 \leq \frac{2 C^2 \left( F(x_0) - F^* \right) }{M_{\min} (K+1)}.     
\end{align}
\end{thm}

\begin{proof}
Combining Lemmas 1-3, we obtain:
\begin{align}
\label{eq:cr}
  F(x_{k+1}) & \leq F(x_{k}) - \dfrac{M_{\min}}{2} \|x_{k+1} - x_{k}\|^2 \nonumber \\
&  \leq F(x_{k}) - \dfrac{M_{\min}}{2 C^2}  \min_{s_{k+1} \in \partial F(x_{k+1})} \|s_{k+1}\|^2.
\end{align}
Adding this relation from $k=0$ to $k=K$, and using Assumption 1 (ii) we obtain the convergence rate \eqref{conv1}. 
\end{proof}

\noindent Furthermore, improved (local) convergence rates can be derived  for our algorithm when $F$ satisfies additionally the Kurdyka-Lojasiewicz  (KL) property \cite{bolte2007lojasiewicz}, i.e., when there exist $\gamma, \epsilon >0$ such that   $F$ satisfies for all $x$ with $\text{dist}(x, \mathcal{X}) \leq \gamma$ and $F_* < F(x) < F_* + \epsilon$ the following inequality  
	\begin{align}
    \label{kl}
		& 	F(x) - F_*  \leq \sigma_q \text{dist}(0, \partial  F(x))^q,  
	\end{align}
where  $\mathcal{X}$ is a given compact set  on which $F$ takes a constant value $F_*$, $\sigma_q>0$ and  $q>1$. The KL property \eqref{kl} holds for a large class of functions including semi-algebraic functions (e.g., real polynomial functions or functions whose graph  is given as a finite union of sets defined by a finite number of polynomial equalities and inequalities);   piecewise linear/quadratic functions such as $\Vert x\Vert _{1}, \Vert x\Vert _{0}, \gamma\sum_{i=1}^{k}|x_{[i]}|$,  where $|x_{[i]}|$ is the $i$-th largest  entry in $x, \;k\leq n$ and $\gamma \in (0, 1]$; the indicator function $\delta_{\Delta}(x)$, where $\Delta= \{x\in\mathbb{R}^n : e^T x = 1, x \geq 0\}$;  least-squares problems with the smoothly clipped absolute deviation,  see  \cite{NecCho:24, Teb:14}. In the next theorem, we state the improved (local) convergence rates under the KL property.
\begin{thm}
\label{thm:kl-decrease}
Let Assumption~1 hold, and assume that the sequence of iterates $({x_k})_{k\geq 0}$ generated by algorithm~RCCM is bounded. Additionally, we assume that $F$ satisfies the KL property \eqref{kl} and  that there exists $k_0\geq 0$ such that, for every $k\geq k_0$, the iterates $x_{k}$ belong to this KL neighborhood. Then, for every $k\geq k_0$, we have:
\begin{align}
\label{eq:kl-decrease-final}
F(x_{k+1}) \!- \!F^*
\!\leq\!
(F(x_k)
-F^*)\! -\!
\frac{M_{\min}}{2C^2}
\left(\!
\frac{F(x_{k+1})-F_*}{\sigma_q}
\!\right)^{2/q}  .
\end{align}
\end{thm}

\begin{proof}
From the sufficient-decrease estimate obtained by combining
Lemma~1 and Lemma~3, we have:
\begin{align}
\label{eq:descent-kl-proof}
F(x_{k+1})
\leq
F(x_k)-
\frac{M_{\min}}{2C^2}
\min_{s_{k+1}\in\partial F(x_{k+1})}
|s_{k+1}|^2.
\end{align}
Since we have
$\operatorname{dist}(0,\partial F(x_{k+1}))
:=
\min_{s_{k+1}\in\partial F(x_{k+1})}
|s_{k+1}|$,
inequality \eqref{eq:descent-kl-proof} becomes
\begin{align}
F(x_{k+1})
\leq
F(x_k)
-
\frac{M_{\min}}{2C^2}
\operatorname{dist}(0,\partial F(x_{k+1}))^2.
\label{eq:descent-distance}
\end{align}
Further, for $k\geq k_0$, the iterate $x_{k}$ lies in the neighborhood
where the KL inequality \eqref{kl} holds, i.e.: 
\begin{align*}
F(x_{k+1})-F_*
\leq
\sigma_q
\left(\operatorname{dist}(0,\partial F(x_{k+1}))^2\right)^{\frac{q}{2}}.
\end{align*}
Substituting the above relation into \eqref{eq:descent-distance}, we obtain
\begin{align*}
F(x_{k+1})
\leq
F(x_k)
-
\frac{M_{\min}}{2C^2}
\left(
\frac{F(x_{k+1})-F_*}{\sigma_q}
\right)^{2/q},
\end{align*}
which proves \eqref{eq:kl-decrease-final}.
\end{proof}

\noindent It is worth noting that the convergence rate depends on the KL parameter $q >1$. More precisely, let us denote:
\[
\Delta_k:=F(x_k)-F_*.
\]
Then, for all sufficiently large $k$, equation \eqref{eq:kl-decrease-final} becomes:
\begin{align}
\Delta_{k+1}+\mu \Delta_{k+1}^{2/q}\leq \Delta_k,
\qquad
\mu:=\frac{M_{\min}}{2C^2\sigma_q^{2/q}}>0.
\label{eq:rec}
\end{align}
When $q\in(1, 2)$, we have $\frac{2}{q}>1$ and the  recursion \eqref{eq:rec} yields a sublinear convergence rate in the objective residual of the form
\[
\Delta_k=\mathcal{O}\left(k^{-\frac{q}{2-q}}\right) \quad \forall k \geq k_0.
\]
Note that the sublinear rate in objective residual improves as the KL parameter $q$ approaches $2$ from below. When $q=2$, the recursion \eqref{eq:rec} becomes
\[
(1+\mu)\Delta_{k+1}\leq \Delta_k,
\]
and therefore the objective values converge linearly:
\[
\Delta_k\leq
\left(\frac{1}{1+\mu}\right)^{k-k_0}\Delta_{k_0} \quad \forall k \geq k_0.
\]
When $q>2$, we have $\frac{2}{q}<1$. In this case, the recursion \eqref{eq:rec} implies a locally superlinear rate in the objective residual. Indeed, for sufficiently small $\Delta_k$, we obtain
\[
\Delta_{k+1}
\leq
\left(\frac{\Delta_k}{\mu}\right)^{q/2},
\]
where $q/2>1$. Consequently, once the iterates enter the KL neighborhood, the objective residual  decreases faster than linearly, i.e., superlinearly. Next, we apply our RCCM algorithm  on a specific matrix
factorization application.


\section{Quantized matrix factorization using RCCM algorithm}
\noindent Quantized matrix factorization (QMF) is a technique that combines matrix factorization with quantization (i.e., elements constrained to bounded integer values) to create low-rank, compressed representations of data, often used for tasks like image compression or efficient training of machine learning models. Traditional compression methods  usually represent the compressed  data in continuous domains and, therefore they  implement quantization steps separately  at the end of the compression process. Hence,  quantization errors cannot be incorporated into the iterative compression process.  Recent work addresses this issue, by introducing  quantization constraints into the matrix factorization problem, see e.g. \cite{Fre:14, qmf}. 

\medskip

\noindent Let us define the QMF problem. Given a data matrix $Y \in \mathbb{R}^{m \times n}$ (e.g., an image),  one approach to tackle the data compression task involves using a low-rank approximation:
\begin{align*}
    Y \approx UV^T = \sum_{j=1}^{d} u^{(j)}\left(v^{(j)}\right)^T 
\end{align*}
where $U \in \mathbb{R}^{m \times d}$, $V \in \mathbb{R}^{n \times d}$, with $d\leq \min(n,m)$ is the rank of the approximation,   and  $u^{(j)}$ and  $v^{(j)}$  denote the  $j^{\text{th}}$ columns of $U$ and $V$, respectively.  The optimization framework assumes minimizing the approximation error quantified via Frobenius norm: 
$$ \min_{U,V} \frac{1}{2}\| Y - UV^T\|^2_F.$$ This is  known as the matrix factorization problem and a globally optimal rank-\(d\) solution can be obtained from the truncated singular value decomposition  (SVD) \cite{eckart1936approximation}. To enhance  transmission, computations and storage efficiency in data compression, a quantization procedure is integrated into the matrix factorization problem: 
\begin{align}
	\label{eq:qmf-l}
	&\!\!\!\!  \min_{ \substack{ U \in \mathbb{Z}_{[\alpha,\beta]}^{m \times d}, V \in \mathbb{Z}_{[\alpha,\beta]}^{n \times d} \\  \alpha, \beta \in \mathbb{Z}}} \frac{1}{2} \| Y - UV^T\|_{F}^{2}  +  \lambda  \log_2(\beta -\alpha +1),
\end{align}
where  the first term  is the distortion, the second is  the rate and  $\lambda \geq 0$ is a Lagrange multiplier that adjusts the weight between output quality and its size. If  $\lambda = 0$, then   $\alpha<\beta$ are fixed integer numbers;  otherwise they are  chosen adaptively.  Formulation 	\eqref{eq:qmf-l} with $\lambda=0$ has been also considered in  \cite{Fre:14, qmf}.  For simplicity of the exposition we also consider  $\lambda=0$ in the sequel, i.e.,  problem  \eqref{eq:qmf-l} with $\alpha<\beta$ fixed and refer to it as the quantized matrix factorization (QMF) problem:
\begin{align}
\label{eq:qmf}
& \min_{ U \in \mathbb{Z}_{[\alpha,\beta]}^{m \times d}, V \in \mathbb{Z}_{[\alpha,\beta]}^{n \times d}} \frac{1}{2} \| Y - UV^T\|_{F}^{2} \\
& =\!\!\!\min_{ U \in \mathbb{R}^{m \times d}, V \in \mathbb{R}^{n \times d}} \! \frac{1}{2} \| Y - UV^T\|_{F}^{2}  + 1_{ \mathbb{Z}_{[\alpha,\beta]}^{m \times d}} (U) + 1_{\mathbb{Z}_{[\alpha,\beta]}^{n \times d}}(V),  \nonumber 
\end{align}
Constraining $U$ and $V$ to have integer entries in a specified interval $[\alpha,\beta]$, the quantization error is incorporated in the compression process, not at the end of the process. Clearly,  the QMF problem \eqref{eq:qmf}  is a particular case of the nonconvex composite optimization problem  \eqref{eq:prob} with
 $x=(U,V)$, twice differentiable nonconvex function  
\begin{align}
\label{eq:f}
f(x) =  \frac{1}{2} \| Y - UV^T\|_{F}^{2}
\end{align} 
and the separable  nonconvex function  
\begin{align}
\label{eq:g}
g(x) =  1_{ \mathbb{Z}_{[\alpha,\beta]}^{m \times d}} (U) + 1_{\mathbb{Z}_{[\alpha,\beta]}^{n \times d}}(V).
\end{align} 
Since the indicator function $g$ has $\text{dom} g$  in a compact set, it follows that all the conditions from Assumption 1 hold for QMF problem \eqref{eq:qmf}.  Notably, for the general problem, boundedness of the iterates is assumed; for QMF, boundedness follows automatically from the integer box constraints. In the QMF problem we fix the block variables $x^{(i)}$ either the columns  $u^{(j)}$ or $v^{(j)}$ of the matrices $U$ and $V$, respectively.   Hence, the previous RCCM algorithm  can be used to solve the QMF problem \eqref{eq:qmf}. 
After the initialization phase, RCCM algorithm requires that, at iteration $k$, we fix $V_k$ and we update the columns of  $U$, then proceed analogously by fixing $U_{k+1}$ and updating the columns of $V$. Moreover, we perform the matrix factor updates columnwise, e.g., examining one column $i_k=1:d$ of $U$
at a time:
\begin{align}
\label{eq:sub_prob_u}
u_{k+1}^{(i_k)}\! = \!\argmin_{u^{(i_k)} \in \mathbb{R}^{m}} \! &\frac{1}{2}\|A^{(i_k)} - \! u^{(i_k)} (v_k^{(i_k)})^T\|^2_F  + 1_{\mathbb{Z}_{[\alpha,\beta]}^{m}} (u^{(i_k)}) \nonumber\\
&+ \frac{M_u}{2} \| u^{(i_k)} -u_k^{(i_k)} \|^2,
\end{align}
where the regularization parameter $M_u \geq 0$,   $A^{(i_k)} = Y - \sum_{j \neq i_k, j=1}^d u_{k, i_k -1}^{(j)}  (v_k^{(j)})^T $ and  $u_{k, i_k -1}^{(j)}$ is the $j^{\text{th}}$ column of matrix $U_{k, i_k-1}$,  which comprises a mixture of both the current and the previously updated columns, i.e $U_{k, i_k-1} := [u_{k+1}^{(1)}, \ldots, u_{k+1}^{(i_k -1)}, u_{k}^{(i_k)}, \ldots , u_{k}^{(d)}]$, with the convention $U_{k, 0} = U_k:= [u_{k}^{(1)},\ldots, u_{k}^{(d)}]$. 
The columns in V are updated solving a similar problem as in \eqref{eq:sub_prob_u} with the regularization parameter $M_v \geq 0$. Hence, in RCCM algorithm we need to solve one of the following generic subproblems:
\begin{align}   
&\min_{u^{(i)} \in \mathbb{Z}_{[\alpha,\beta]}^{m}} \!\! \frac{1}{2} \|A^{(i)} \!- u^{(i)} (\bar v^{(i)}))^T \|_F^2 + \frac{M_u}{2} \| u^{(i)} \!- \bar{u}^{(i)} \|^2 \label{eq:generic_pb_1} \\
&\min_{v^{(i)} \mathbb{Z}_{[\alpha,\beta]}^{n}} \!\!  \frac{1}{2} \| B^{(i)} \!- \bar u^{(i)} (v^{(i)}))^T \|_F^2 + \frac{M_v}{2} \| v^{(i)} \!- \bar{v}^{(i)} \|^2, \label{eq:generic_pb_2}
\end{align}
where the vectors $\bar{u}^{(i)},\bar{v}^{(i)}$ and the matrices $ A^{(i)}, B^{(i)}$ are fixed.  Further, we show how the generic subproblems \eqref{eq:generic_pb_1}  and \eqref{eq:generic_pb_2} can be decomposed into one dimensional  optimization problems that have closed form solutions and consequently step \eqref{eq_subpr} in RCCM algorithm can be applied efficiently for solving the QMF problem.  Indeed, the objective function in \eqref{eq:generic_pb_1} can be rewritten as:
\begin{align*}
     &\frac{1}{2}\|A^{(i)} - u^{(i)} (\bar v^{(i)}))^T \|_F^2 + \frac{M_u}{2} \| u^{(i)} - \bar{u}^{(i)} \|^2 = \\
     & \frac{1}{2} \sum_{p=1}^{m} \sum_{q=1}^{n} \left( A_{p,q}^{(i)} -u^{(i)}_p \bar v^{(i)}_q \right)^2 +\ \ \frac{M_u}{2} \sum_{p=1}^{m}(u^{(i)}_p -\bar{u}^{(i)}_p)^2, 
\end{align*}
where $A_{p,q}^{(i)}$ denotes the $(p, q)$ element  of matrix $A^{(i)}$ and  $u_p^{(i)} $ and $\bar{u}_p^{(i)} $ are the $p$th entries of the vectors $u^{(i)}$ and $\bar{u}_p^{(i)}$, respectively.  Similarly, $v_q^{(i)} $ is the $q$th element of vector $v^{(i)}$.  Since all variables are fixed with the exception of $u_p^{(i)}$, the subproblem \eqref{eq:generic_pb_1} can be divided into $m$ one-dimensional  quadratic subproblems of the  form:
\begin{align} \label{eq:scalar_pb}
    \min_{u^{(i)}_p \in \mathbb{Z}_{[\alpha,\beta]}} \! \frac{1}{2} \! \sum_{q=1}^{n} \! \left( \! A_{p,q}^{(i)} \! -u^{(i)}_p \bar v^{(i)}_q \!\right)^2 \!+\! \frac{M_u}{2} (u^{(i)}_p \! -\bar{u}^{(i)}_p)^2,\,
\end{align}
for $p\!=\!1\!:\!m$, whose solution can  be  computed explicitly~as:
\begin{align}
\label{eq:explicit}
& (u^{(i)}_p)^+ =\left[   \frac{M_u \bar{u}^{(i)}_p + \sum_{q=1}^{n}A_{p,q}^{(i)} \bar v^{(i)}_q}{ M_u + \sum_{q=1}^{n}(\bar v^{(i)}_q)^2 }  \right]_{Z_{[\alpha, \beta]}}\\
&= \text{clamp} \!\left( \alpha, \beta, \text{round} \left(  \frac{M_u \bar{u}^{(i)}_p + \sum_{q=1}^{n}A_{p,q}^{(i)} \bar v^{(i)}_q}{M_u + \sum_{q=1}^{n}(\bar v^{(i)}_q)^2 }  \right ) \!\right ) .  \nonumber
\end{align} 
Hence, using the definition of the matrix $A^{(i)}$, the update of the entire column $u^{(i)}$ is as follows:
\begin{align*}  (u^{(i)})^{+} & =  \left[   \frac{M_u \bar{u}^{(i)} + A^{(i)} \bar v^{(i)}}{ M_u + \|\bar v^{(i)}\|^2}  \right]_{Z^m_{[\alpha, \beta]}} \\
& =  \left[   \frac{M_u \bar{u}^{(i)} + Y \bar v^{(i)} - \sum_{s \neq i }  (\bar V^T \bar V)_{s,i} \bar{u}^{(s)}  } { M_u + \|\bar v^{(i)}\|^2 }  \right]_{Z^m_{[\alpha, \beta]}}.    
\end{align*}
The same arguments apply equally to subproblem \eqref{eq:generic_pb_2}.  In conclusion, the application of Algorithm 1 (RCCM)  on  QMF problem \eqref{eq:qmf} leads to   Algorithm \ref{alg:RQMF}. The algorithm in \cite{Teb:14} has been used in \cite{qmf} to solve the QMF   problem \eqref{eq:qmf} and called \text{QMF} algorithm. Hence, we call Algorithm 1 (RCCM) specialized for the  QMF problem \eqref{eq:qmf} as \textit{regularized quantized matrix factorization}  (RQMF) algorithm.  Since $f$ in \eqref{eq:f} is quadratic in $U$ and $V$, respectively,  the main difference between our Algorithm \ref{alg:RQMF} (RQMF) and the proximal alternating  linearized minimization algorithm in \cite{Teb:14} (which is a projected coordinatewise gradient) when applied to the particular QMF problem \eqref{eq:qmf} consists of the regularization terms for $U$ and $V$. Hence, for $M_u=M_v =0$ our algorithm recovers the one in  \cite{Teb:14} for solving the QMF problem. However, the regularization parameters $(M_u,M_v)$ provide  more flexibility to our Algorithm \ref{alg:RQMF}, guaranteeing \textit{strict} decrease of the objective function along iterations (see Lemma 1) while  the algorithm in \cite{Teb:14} ensures only nonincresing monotonicity (see Theorem 1 in \cite{qmf}), and leads to improved numerical results (see next~section). 

\begin{algorithm}
\caption{(RQMF)} 
\label{alg:RQMF}
\begin{algorithmic} 
          \STATE Input: $Y \in \mathbb{R}^{m \times n}$, $M_u,\, M_v \in \mathbb{R}_+$, rank $d$, bounds $[\alpha, \beta]$ and maximum number  of  iterations $K$.
			 \STATE Set $k=0$. Compute ${\bar{U}}_{0} \in  \mathbb{R}^{m \times d}, \bar{V}_{0} \in \mathbb{R}^{n \times d}$ using truncated SVD, then project
             \begin{align*}
                 U_0 = [\bar{U}_0]_{\mathbb{Z}^{m\times d}_{[\alpha, \beta]}}, \text{ and }   V_0 = [\bar{V}_0]_{\mathbb{Z}^{n\times d}_{[\alpha, \beta]}}
             \end{align*}
			\WHILE{$k \leq  K$}
            \STATE Compute: $A = [a^{(1)} \cdots a^{(d)}]= Y V_k$, \;  $C = V_k^TV_k$
			\FOR{$i_{k} = 1:d$}
            \vspace{-0.6cm}
			     \STATE\begin{align*}
			   \text{1. } &\hat{a}^{(i_k)} =   a^{(i_k)} - \left( \sum_{s=1}^{i_k-1} C_{s, i_k} u_{k+1}^{(s)} + \sum_{s=i_k+1}^{d}  C_{s,i_k} u_{k}^{(s)} \right)\\
                  \text{2. } &u^{(i_k)}_{k+1} = \left[  \frac{M_u u_{k}^{(i_k)} + \hat a^{(i_k)}}{M_u + \| v_{k}^{(i_k)}\|^2}  \right]_{\mathbb{Z}^{m}_{[\alpha, \beta]}}
			\end{align*}
            \ENDFOR  
            \STATE Compute: $B \!=\! [b^{(1)} \cdots b^{(d)}] \!=\! Y^T U_{k+1}, \;  C = U_{k+1}^TU_{k+1}$
			\FOR{$i_{k} = 1:d$}
            \vspace{-0.6cm}
			\STATE \begin{align*}
			    \text{1. } &\hat{b}^{(i_k)} =  b^{(i_k)} - \left( \sum_{s=1}^{i_k-1} C_{s, i_k} v_{k+1}^{(s)} + \sum_{s=i_k+1}^{d}  C_{s,i_k} v_{k}^{(s)} \right)\\
                  \text{2. } &v^{(i_k)}_{k+1} = \left[  \frac{M_v v_{k}^{(i_k)} + \hat b^{(i_k)}}{M_v + \| u_{k}^{(i_k)}\|^2}  \right]_{\mathbb{Z}^{n}_{[\alpha, \beta]}}
			\end{align*}
                  \ENDFOR  
			 \STATE $k \leftarrow k+1$
			\ENDWHILE
		\end{algorithmic} 
	\end{algorithm}

\noindent From \eqref{eq:explicit} it is clear that Algorithm \ref{alg:RQMF} (RQMF) has very simple  iterations (involving only matrix-vector operations).

\begin{remark}
(i) One can notice that the fourth-order polynomial function $f$ defined in \eqref{eq:f} is twice differentiable and $\text{dom} g$ of the separable function $g$ defined in \eqref{eq:g} is compact. Therefore,  all the conditions from Assumption 1 hold for QMF problem \eqref{eq:qmf}. Moreover,  the iterates generated by (RQMF) $x_k=(U_k, V_k)$ are bounded as they lie in the interval $[\alpha,\beta]$. In conclusion, the sublinear  convergence result from Theorem 1 is valid for this particular Algorithm \ref{alg:RQMF} (RQMF) on solving the QMF problem \eqref{eq:qmf}.

\vspace{0.15cm}

\noindent (ii) Additionally, function $f$ defined in \eqref{eq:f} is a fourth-order polynomial, hence semi-algebraic;  the interval $[\alpha,\beta]$ intersected with the set of integers is semi-algebraic, hence  the indicator functions that define  $g$ in \eqref{eq:g} are also semi-algebraic. Therefore, the objective function of the QMF problem \eqref{eq:qmf},   $F=f+g$,  is semi-algebraic, and thus $F$ has the KL property.  In conclusion, the convergence results from Theorem 2 are also valid for this particular Algorithm \ref{alg:RQMF} (RQMF) on solving the QMF problem \eqref{eq:qmf}.  
\end{remark}


\section{Numerical experiments on image compression}
\noindent In this section we present numerical results on image compression using real data. We evaluate our proposed Algorithm \ref{alg:RQMF} (RQMF), against several baselines: Singular Value Decomposition (SVD) based compression \cite{golpayegani2024patchsvd}, standard JPEG \cite{wallace1991jpeg},   and quantization-aware matrix factorization (QMF) algorithm \cite{qmf}. For RQMF and QMF algorithms we choose $K=10$ full iterations. For simulations we use three datasets: KODAK (24 lossless images with a resolution of $768\times 512$) \cite{kodak1993}, CLIC 2024 (32 high-resolution, high-quality images) \cite{clic2024} and ImageNet validation dataset (50.000 images with various resolutions, all higher than $256 \times 256$) \cite{deng2009imagenet}. Then,  the compressed ImageNet dataset is also used  in a classification task using 3 well-known convolutional networks AlexNet \cite{alexnet}, ResNet50 \cite{resnet} and MobileNetV2 \cite{mobilenetv2} based on floating point and integer arithmetic representations.  All experimental results were implemented in Python on  a laptop with Intel Core i9 up to 5.4 GHz (14 physical cores and 20 threads) and  32 GB RAM memory.   The code can be found on Github:  \href{https://github.com/tibi-coder/rqmf-image-compression}{rqmf-image-compression}.

\subsection{Preprocessing of color images}
\noindent We follow the same preprocessing workflow as in \cite{qmf}, see also \cite{golpayegani2024patchsvd}. More specifically, for color image compression the recommended color transformation is the $(Y C_B C_R)$ color space instead of the RGB model, as the human visual system is more sensitive to luminance variations than to chrominance differences \cite{wallace1991jpeg, netravali1995digital}. Hence, in our experiments we decompose our initial RGB image $X \in \mathbb{Z}_{[0, 255]}^{m \times n \times 3} $ in 3 matrices: (i) $Y \in \mathbb{R}_{[0,\, 255]}^{m \times n}$ called \textit{luma} which encapsulates the brightness information; (ii) $Y_{C_B},\, Y_{C_R} \in \mathbb{R}_{[0,\, 255]}^{m \times n}$ correspond to the chrominance components encoding the blue-difference and red-difference signals, respectively. Note that in this representation, the values are continuous but restricted to the interval $[0, 255]$. Moreover, we downsample each chrominance channel using an average-pooling kernel of size $(2, 2)$, which results in  $Y_{C_B},\, Y_{C_R} \in \mathbb{R}_{[0,\, 255]}^{\frac{m}{2} \times \frac{n}{2}}$. Due to the nature of our matrix decomposition–based compression method, it is beneficial to use  patching in the preprocessing step \cite{qmf, golpayegani2024patchsvd}. Hence, each channel is decomposed into non-overlapping $8\times 8$ patches.  These patches are then flattened and stacked vertically,  resulting in the new matrices $Y \in \mathbb{R}_{[0,\, 255]}^{\frac{mn}{64} \times 64 }$, $ Y_{C_B},\, Y_{C_R} \in \mathbb{R}_{[0,\, 255]}^{\frac{mn}{256} \times 64}$. When the compression interval $[\alpha, \beta]$ is centered in zero, slightly improved compression results are obtained when these matrices are also centered around zero, e.g., $Y \leftarrow Y - 128$.


\subsection{Experimental setup for image quantized compression}
\noindent For all methods, the integer bounds for the factor matrices were set to $[\alpha, \beta] = [-16, 15]$ and the maximum number of iterations for RQMF and QMF algorithms was set to $K = 10$. We initialize algorithms RQMF and QMF with $U_0, V_0$ from truncated SVD.   The factorization rank, $d$, was controlled by the quality factor $Q \in [0, 1]$, where $d = \max(\text{round}(Q \times \min(m, n)), 1)$. \\

\noindent \textit{Metrics:}  we evaluate the methods by using standard rate-distortion metrics such as Peak Signal-to-Noise Ratio (PSNR), Structural Similarity Index Measure (SSIM) and Bits Per Pixel (bpp). The first two metrics assess quality of  compressed images, while the last metric measures memory usage.\\

\noindent  \textit{Parameters choice for RQMF}: Note that Algorithm \ref{alg:RQMF} (RQMF) uses 2 regularization parameters $M_u$ and $M_v$. Our objective is to minimize the final bit rate (bpp) while maintaining a high PSNR and SSIM. This requires a content-aware trade-off: (i) \textit{high complexity images} (i.e., high contrast, high entropy) demand weak regularization ($M_u, M_v$ small) to preserve fine details and energy; (ii) \textit{low complexity images} (i.e., uniform areas) can tolerate strong regularization ($M_u, M_v$ large), which significantly lowers the entropy of the resulting integer factors, thereby maximizing lossless compression efficiency and minimizing bpp.
Hence, we set the parameters $M_u, M_v$ depending on the entropy of the image:
\begin{align}
\label{eq:M}
M_u = M_v = \gamma \cdot \left( 1 - \frac{H(Y)}{H_{\max}} \right),
\end{align}
where $\gamma$ is a scaling factor, $H(Y)$ is the Shannon entropy of image $Y$ and $H_{\max}$ is constant and represents the maximum entropy. Given the reduced sensitivity of the human visual system to color errors compared to luminance errors, we employ a differentiated strategy (empirically designed content-adaptive heuristic) for regularization of the luminance and chroma channels: \textit{Luminance ($Y$)} - the $\gamma$  is set to a restrictive value ($\gamma = 60$)  to strictly preserve structural details; \textit{Chrominance ($C_R, C_B$)} - a strong, content-adaptive regularization $\gamma$ is applied to maximize the compression efficiency (minimize bpp) in these channels ($\gamma = 80$). 
The term $\left( 1 - H_c/H_{\max} \right)$ quantifies \textit{compresibility}: a high value indicates low information entropy ($H_c$) and high potential for lossless compression. 

\subsection{Numerical results for image quantized compression}

\noindent The numerical results for image quantized compression are reported in Figures \ref{fig2} and  \ref{fig3}. 
\begin{figure}[h]
    \centering
    \includegraphics[width=1\linewidth]{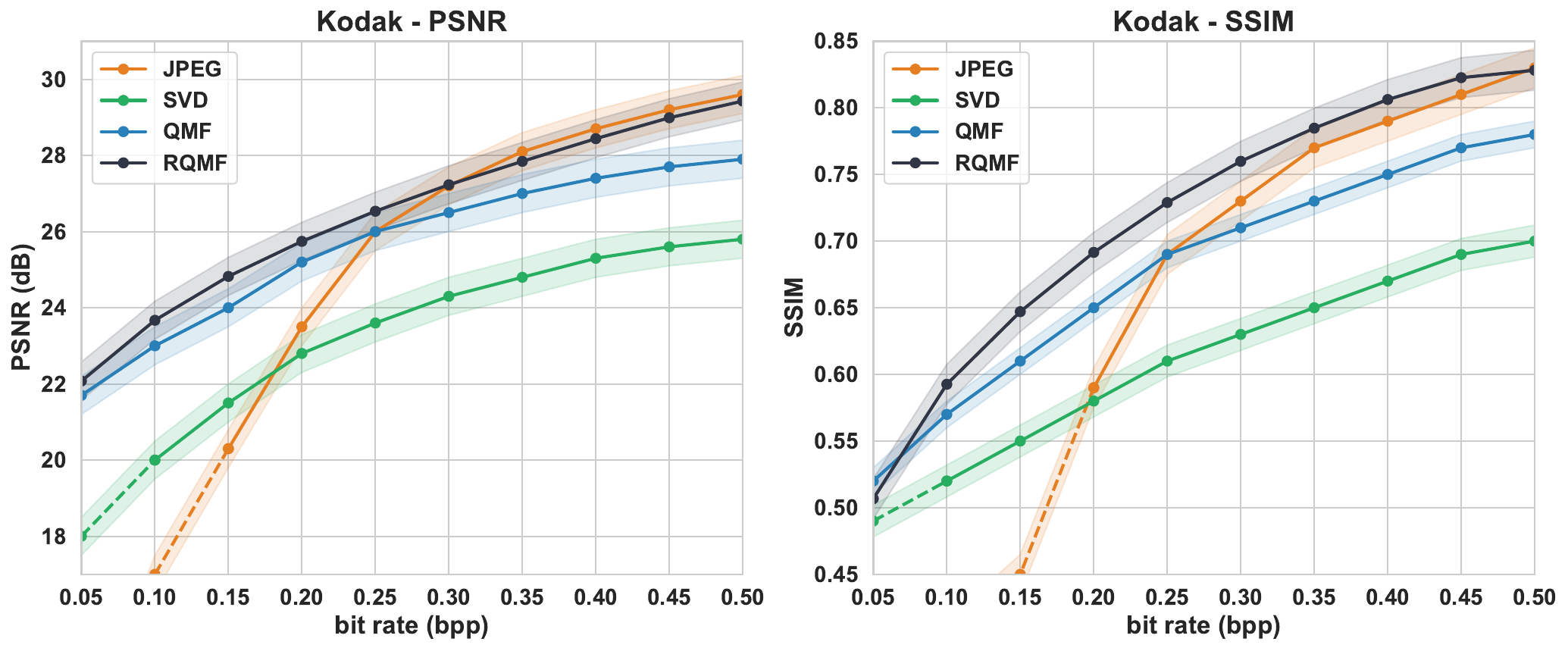}\\
    \includegraphics[width=1\linewidth]{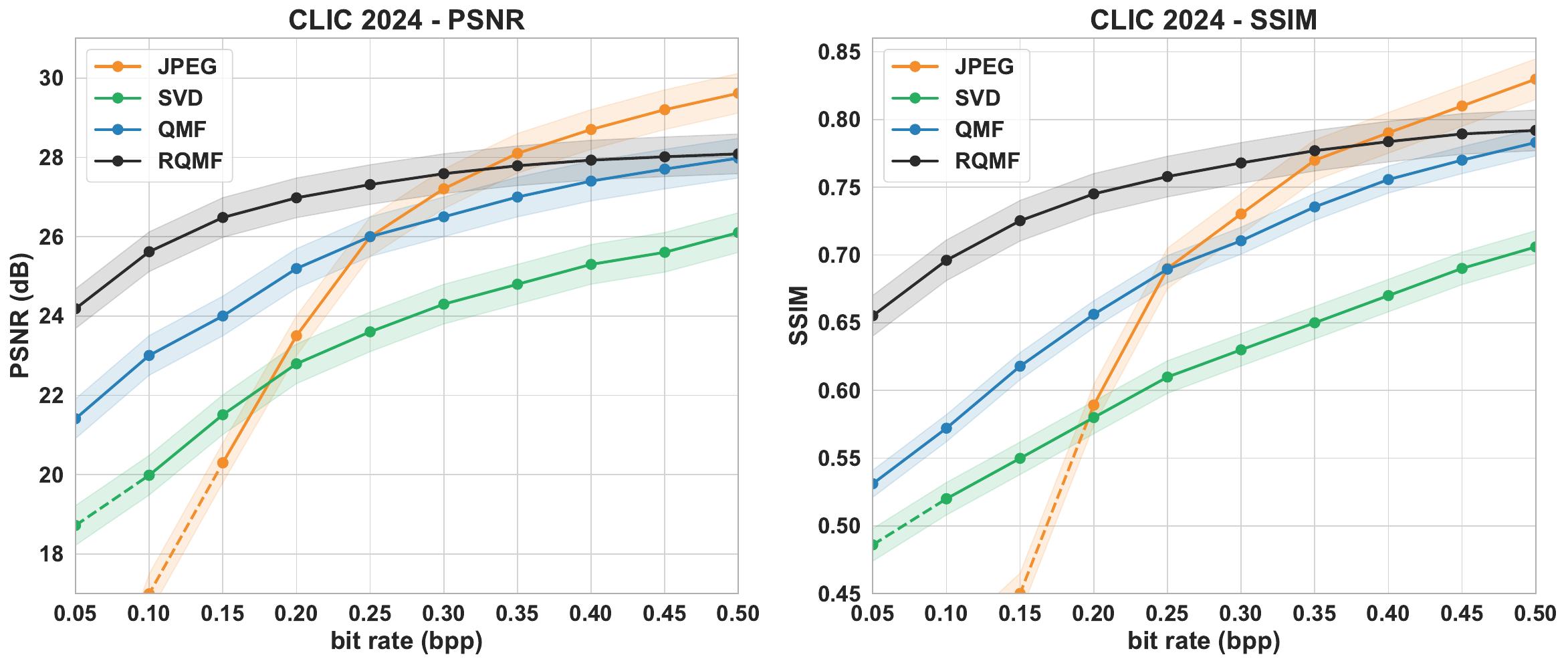}
    \caption{Rate-distortion performance on the Kodak (top) and CLIC 2024 (bottom) datasets. The average PSNR (left) and average SSIM (right) for each method are plotted as functions of bit rate. Shaded areas represent standard errors. Dashed lines indicate extrapolated values predicted using LOESS \cite{cleveland1988locally} for extremely low bit rates that are otherwise unattainable. }
     \label{fig2}
    \end{figure} 
\noindent The rate-distortion curves in terms of PSNR and SSIM versus bpp are presented in Figure \ref{fig2} for our algorithm  RQMF and the baselines QMF, SVD and JPEG. 
As one can see from Figure \ref{fig2}, RQMF consistently achieves superior performance, significantly surpassing  SVD, JPEG and QMF at low bit rates. More specifically,  on the Kodak images, RQMF achieves consistently stronger PSNR and SSIM performance across all bit rates compared to  SVD, JPEG and QMF. At higher bit rates, its PSNR performance is comparable to that of JPEG, while it provides the best overall visual quality (SSIM), as illustrated in the top-right of Figure 1. For high-resolution images, such as  CLIC 2024 dataset, RQMF notably outperforms SVD, JPEG and QMF at low bit rates below $0.3$ bpp, enabling bit rate savings without introducing excessive visual degradation. Across both datasets, our algorithm RQMF achieves PSNR values above 24 dB and SSIM scores above 0.60 for bpp$>\!0.1$, demonstrating that combination of  coordinate minimization and proper regularization yields an algorithm that enjoys convergence guarantees and it is also  efficient,  robust in preserving visual quality and better in several metrics than an algorithm which does not use regularization as QMF \cite{qmf}. 
\begin{figure*}[!t]
    \centering
    \begin{multicols}{3}
        \centering
        \scriptsize
        \setlength{\tabcolsep}{2pt}

        \begin{tabular}{c c c c c}
            \toprule
            \textbf{Orig.} &
            \textbf{SVD} &
            \textbf{JPEG} &
            \textbf{QMF} &
            \textbf{RQMF} \\
            \midrule

           \shortstack{
            \includegraphics[
                width=0.6\columnwidth,
                valign=m
            ]{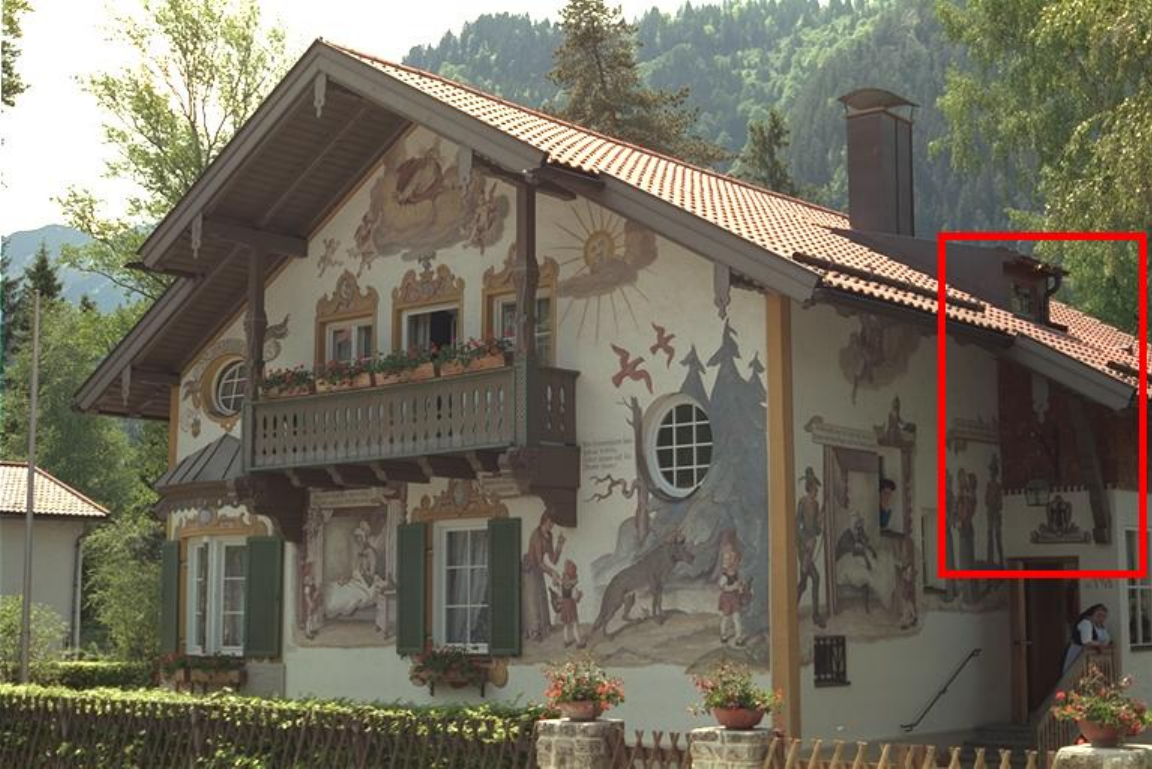} 
            \\
           \textbf{Kodak}
             }
            &
            \shortstack{
                \includegraphics[
                    width=0.6\columnwidth,
                    valign=m
                ]{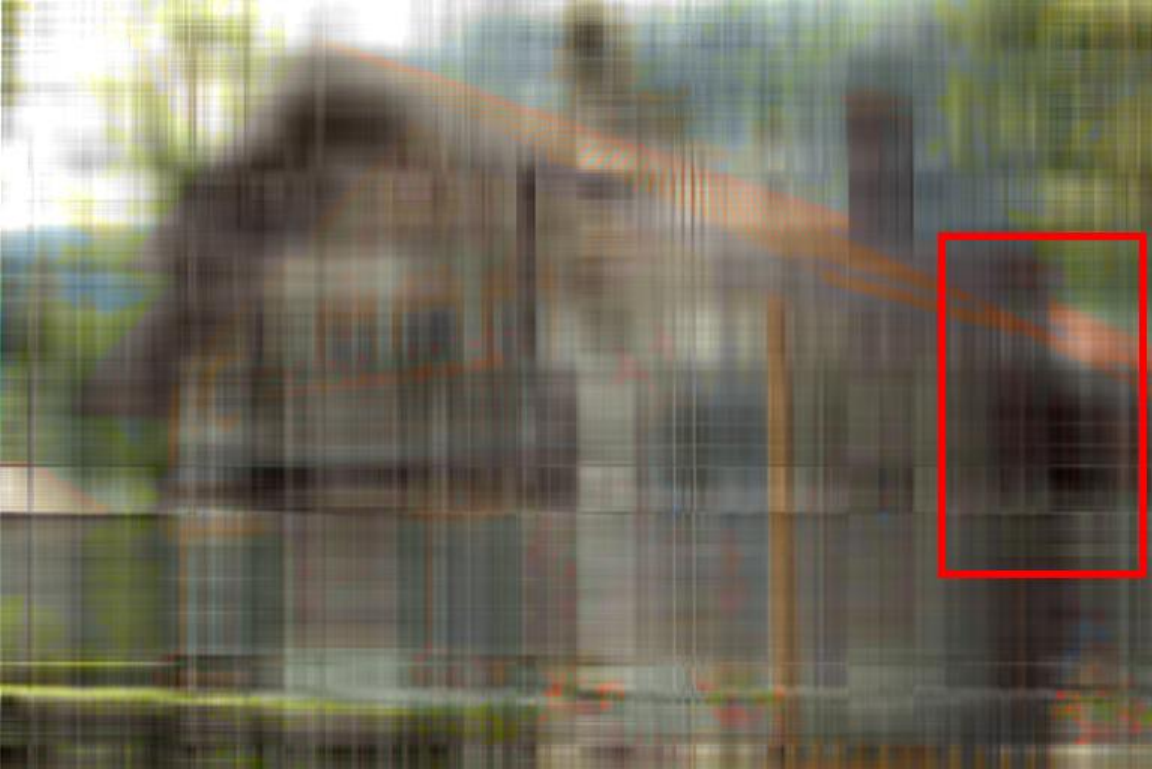}
                \\
                0.26 bpp / 18.82 dB
            }
            &
            \shortstack{
                \includegraphics[
                    width=0.6\columnwidth,
                    valign=m
                ]{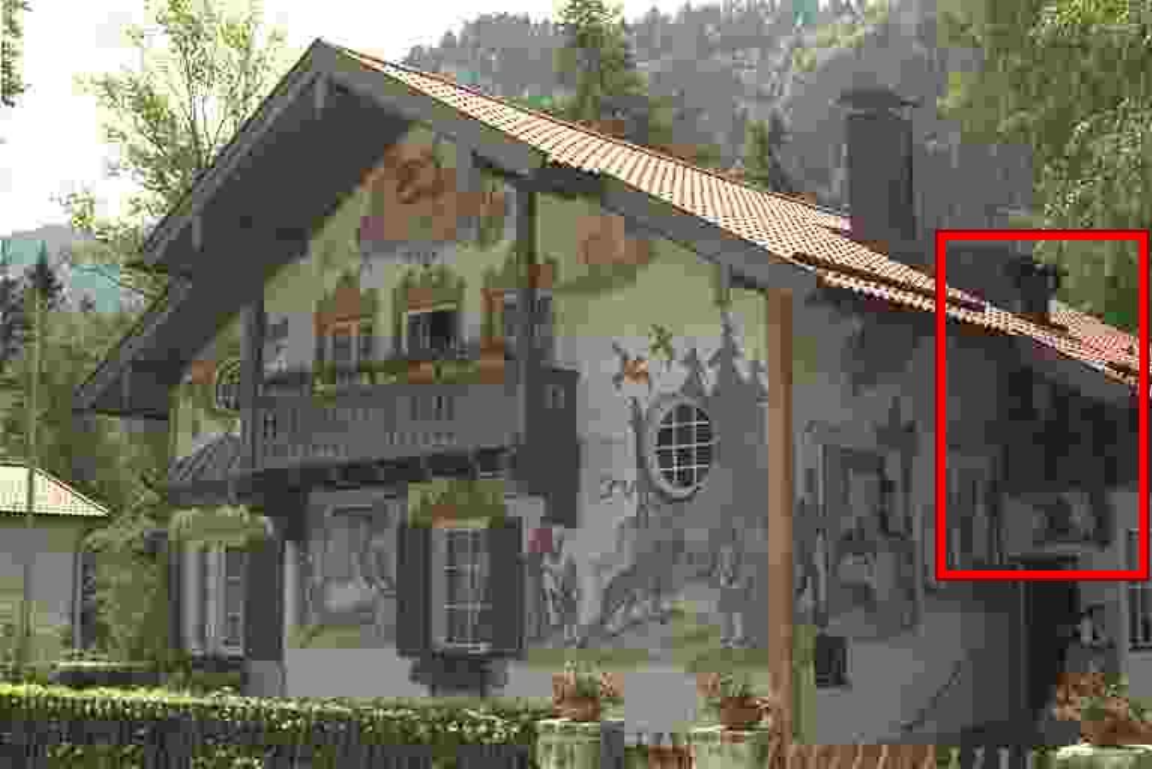}
                \\
                0.27 bpp / 23.16 dB
            }
            &
            \shortstack{
                \includegraphics[
                    width=0.6\columnwidth,
                    valign=m
                ]{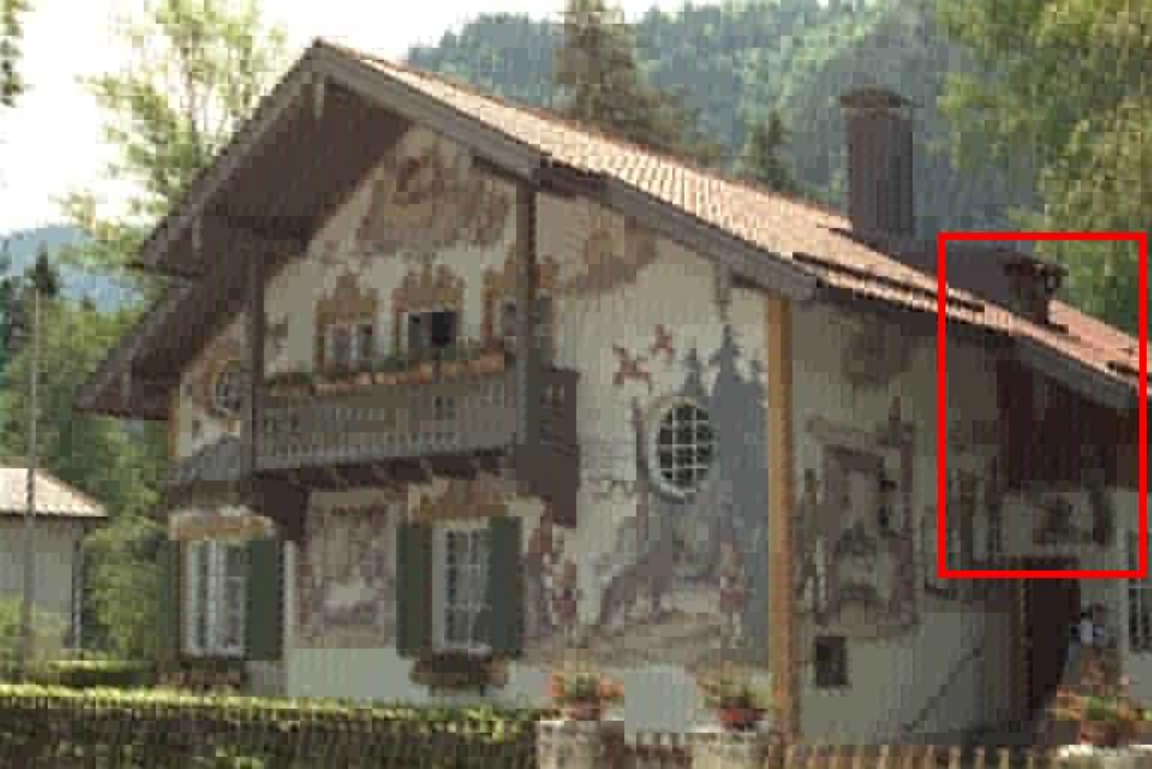}
                \\
                0.26 bpp / 23.04 dB
            }
            &
            \shortstack{
                \includegraphics[
                    width=0.6\columnwidth,
                    valign=m
                ]{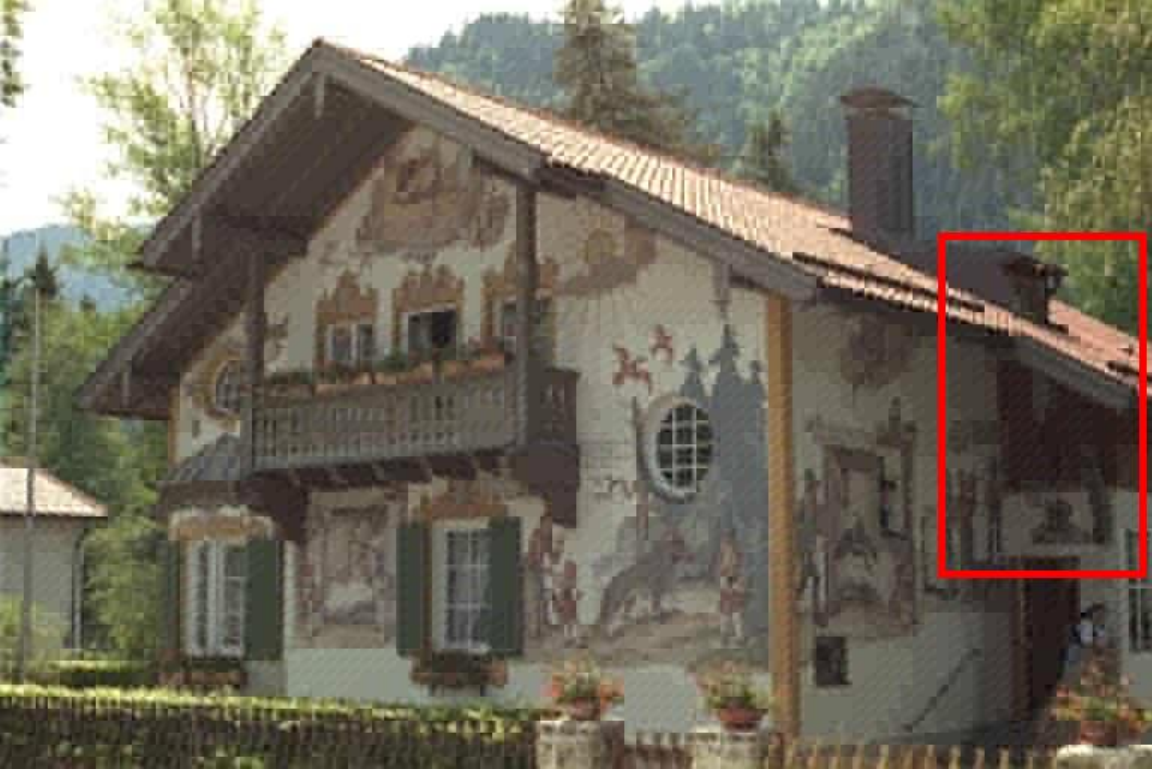}
                \\
                0.25 bpp / 23.40 dB
            }
            \\
            \midrule

            \shortstack{
            \includegraphics[
                width=0.6\columnwidth,
                valign=m
            ]{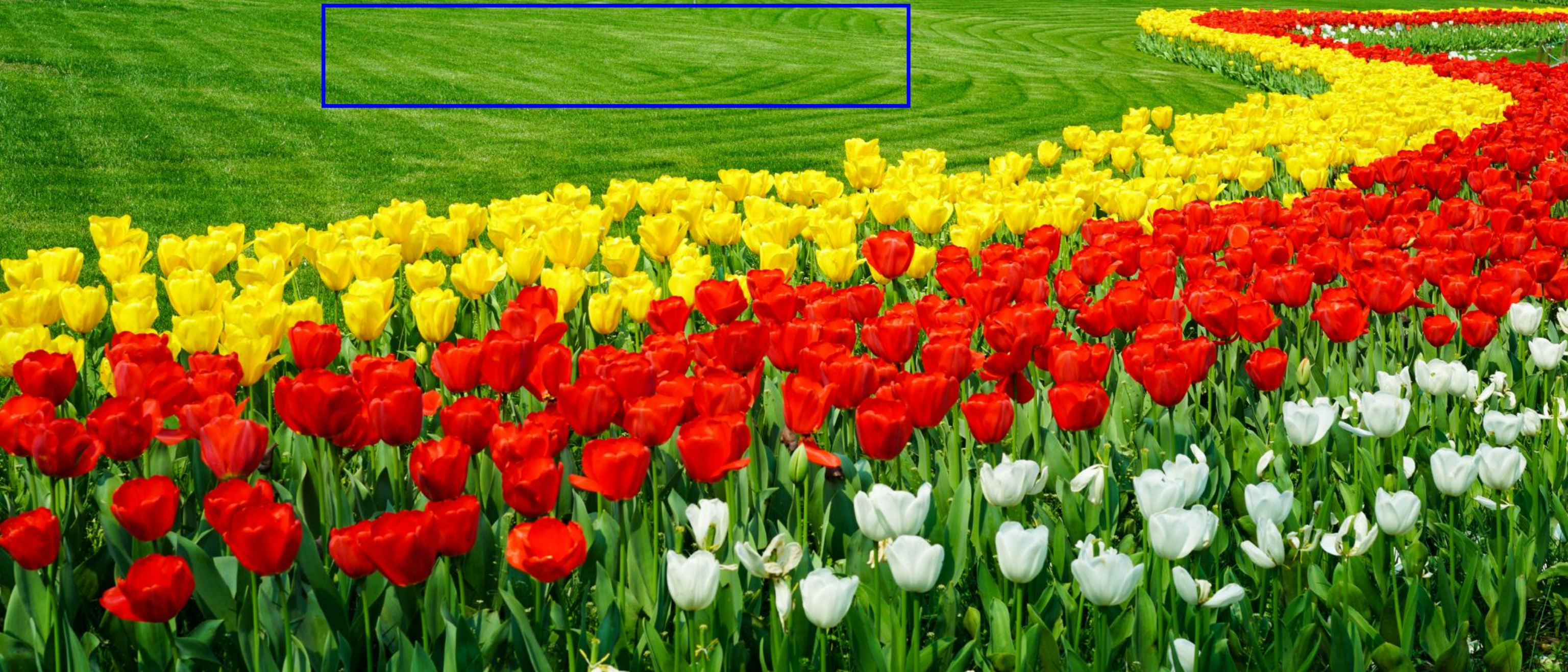}
              \\
              \textbf{CLIC 2024}
             }
            &
            \shortstack{
                \includegraphics[
                    width=0.6\columnwidth,
                    valign=m
                ]{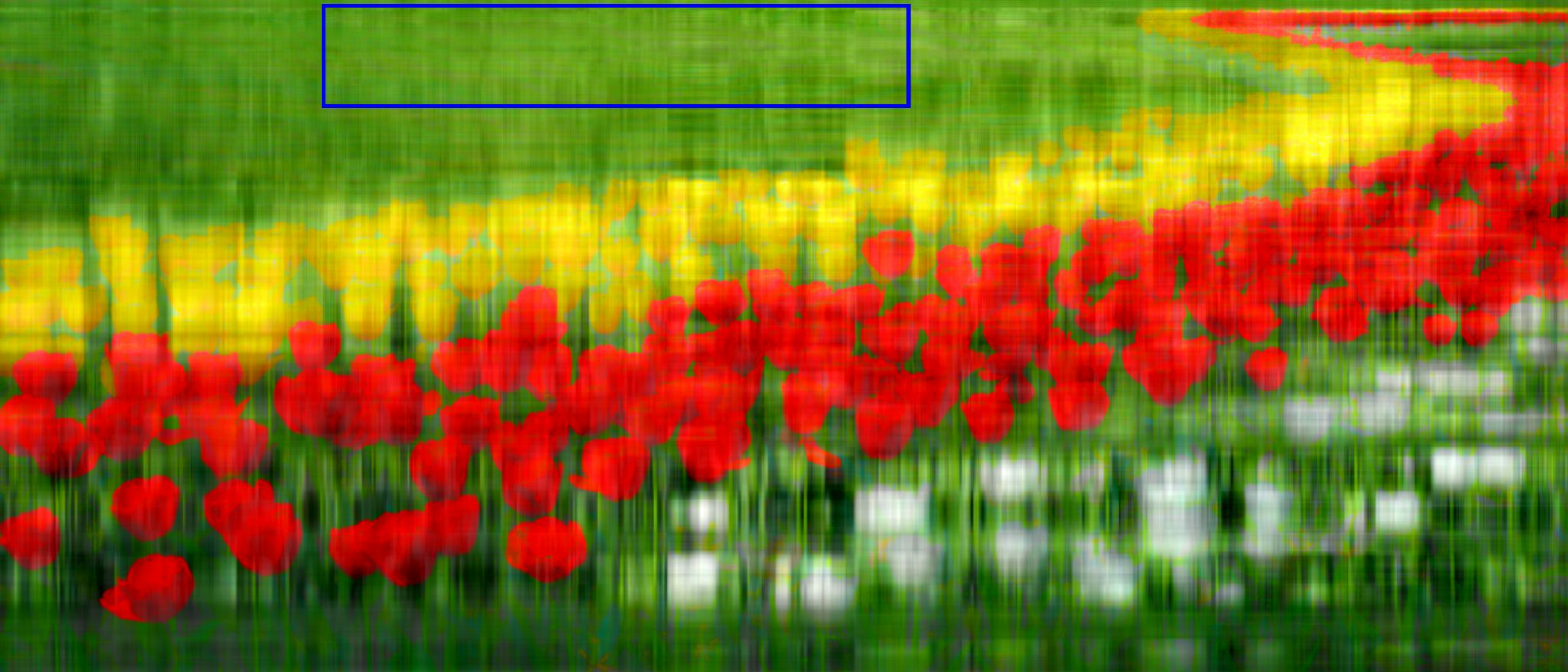}
                \\
                0.18 bpp / 18.01 dB
            }
            &
            \shortstack{
                \includegraphics[
                    width=0.6\columnwidth,
                    valign=m
                ]{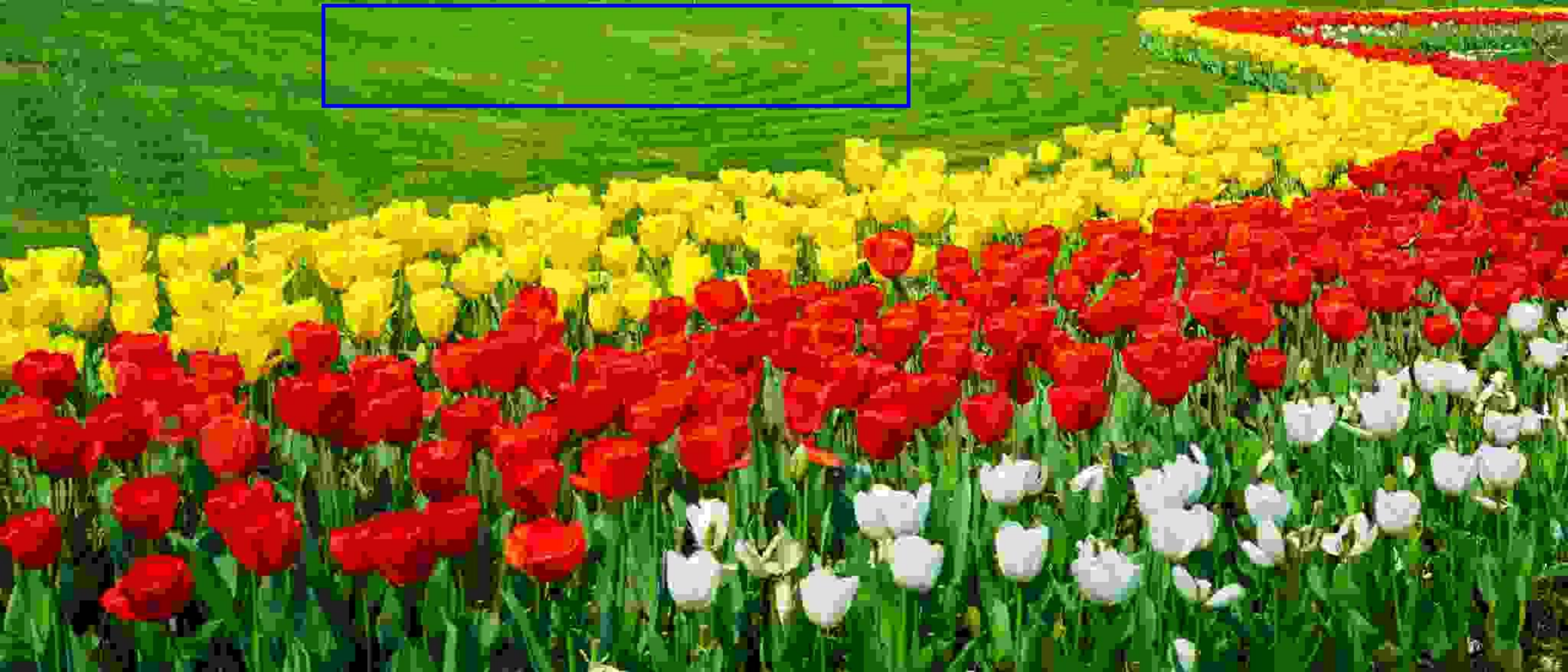}
                \\
                0.18 bpp / 20.40 dB
            }
            &
            \shortstack{
                \includegraphics[
                    width=0.6\columnwidth,
                    valign=m
                ]{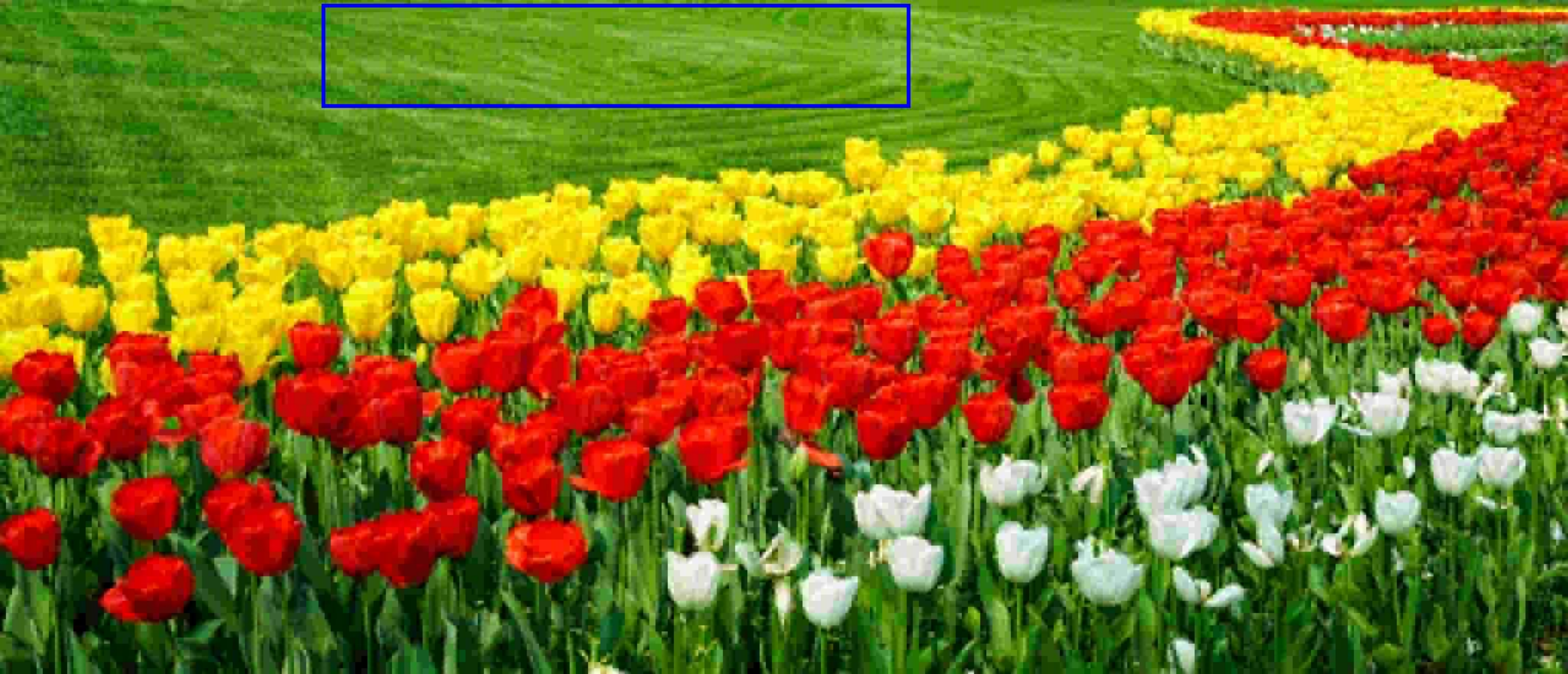}
                \\
                0.17 bpp / 21.93 dB
            }
            &
            \shortstack{
                \includegraphics[
                    width=0.6\columnwidth,
                    valign=m
                ]{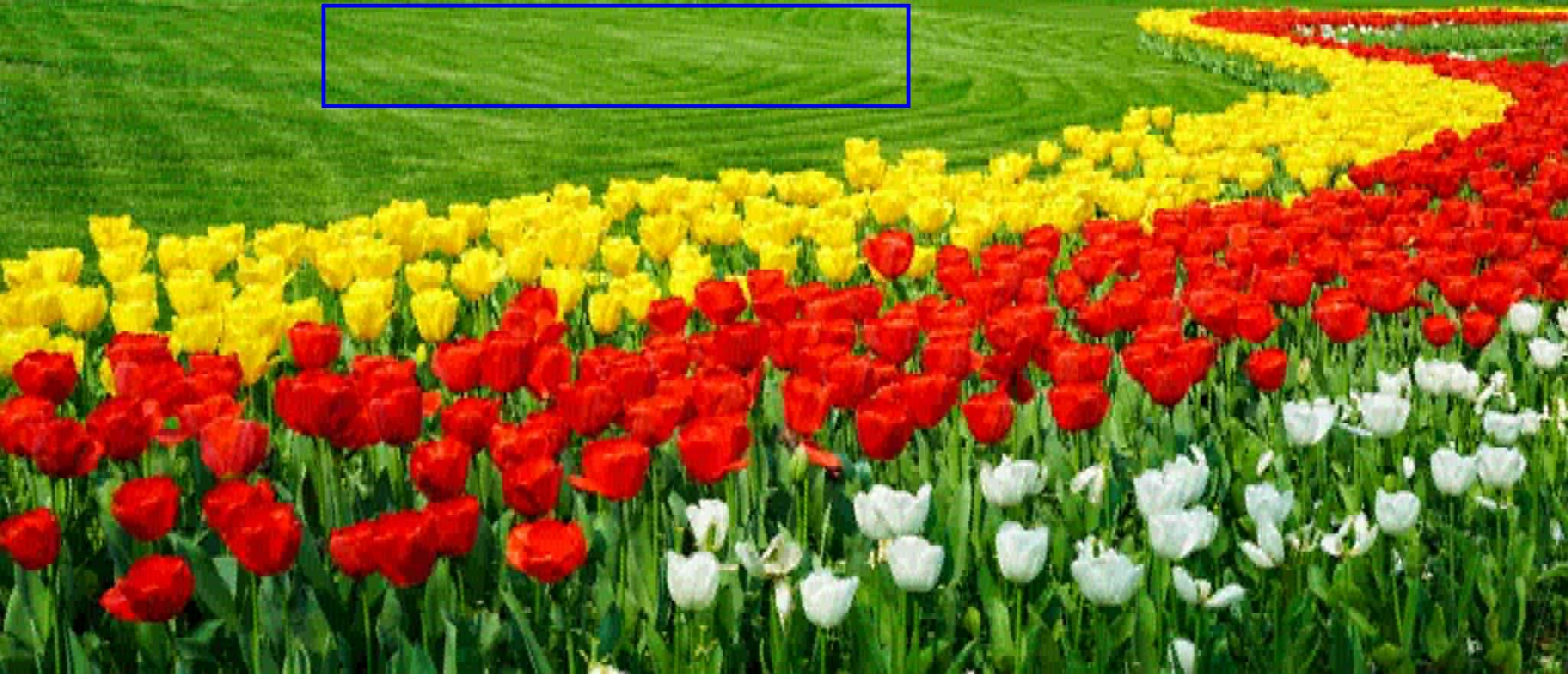}
                \\
                0.17 bpp / 22.51 dB
            }
            \\
            \bottomrule
        \end{tabular}
    \end{multicols}

    \caption{
        Qualitative performance comparison on example images from the
        Kodak and CLIC 2024 datasets. The colored bounding boxes highlight
        the artifacts introduced by the compression methods.
    }
    \label{fig:transposed_comparison}
\end{figure*}

\begin{figure*}[!b]
    \centering
    \includegraphics[width=0.5\columnwidth, height = 0.45\columnwidth]{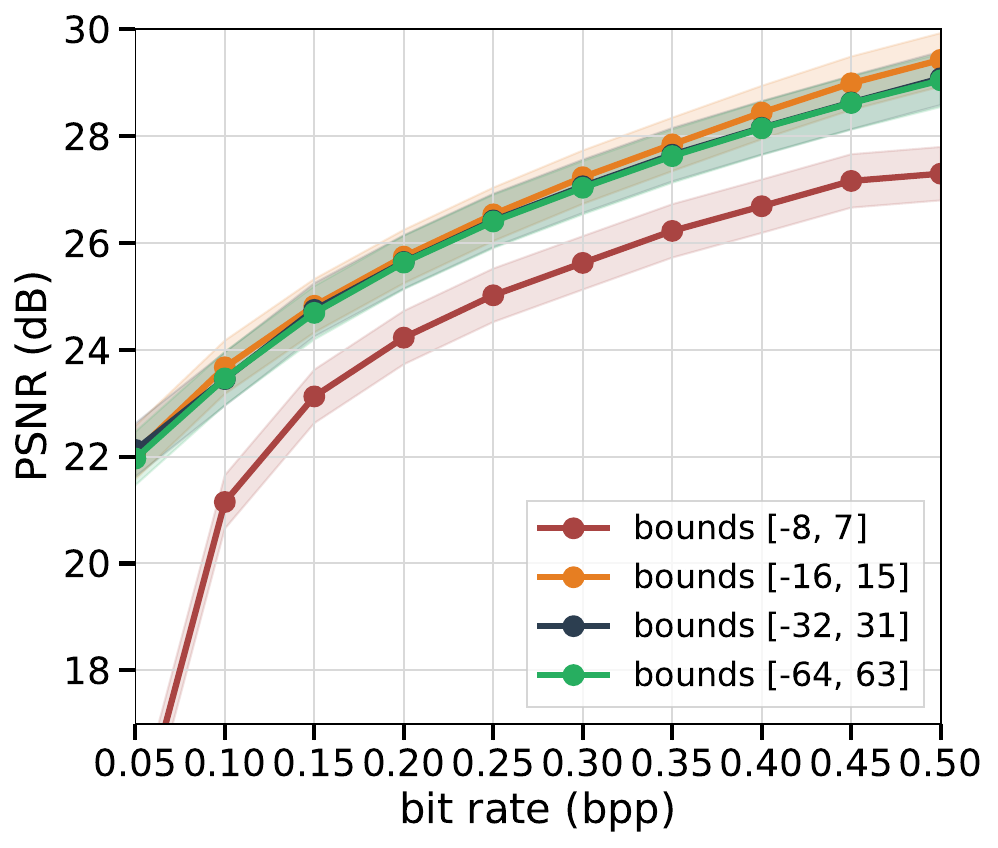} 
\includegraphics[width=0.5\columnwidth, height = 0.45\columnwidth]{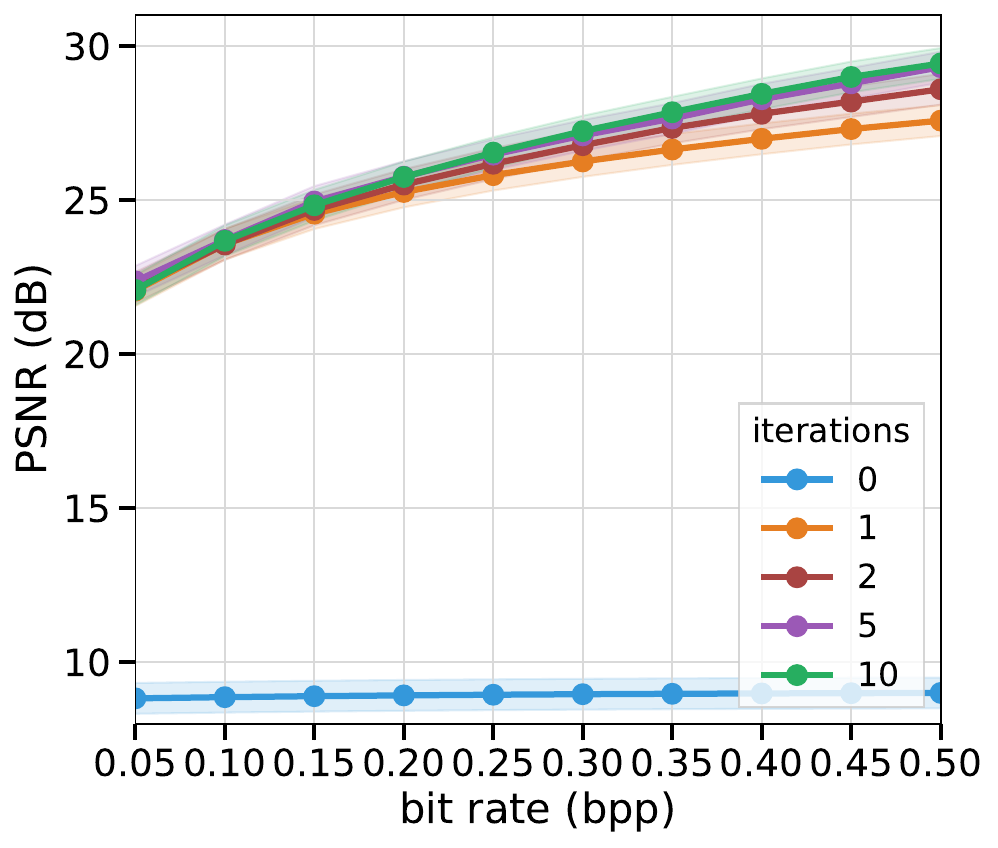}
    \includegraphics[width=0.5\columnwidth, height = 0.45\columnwidth]{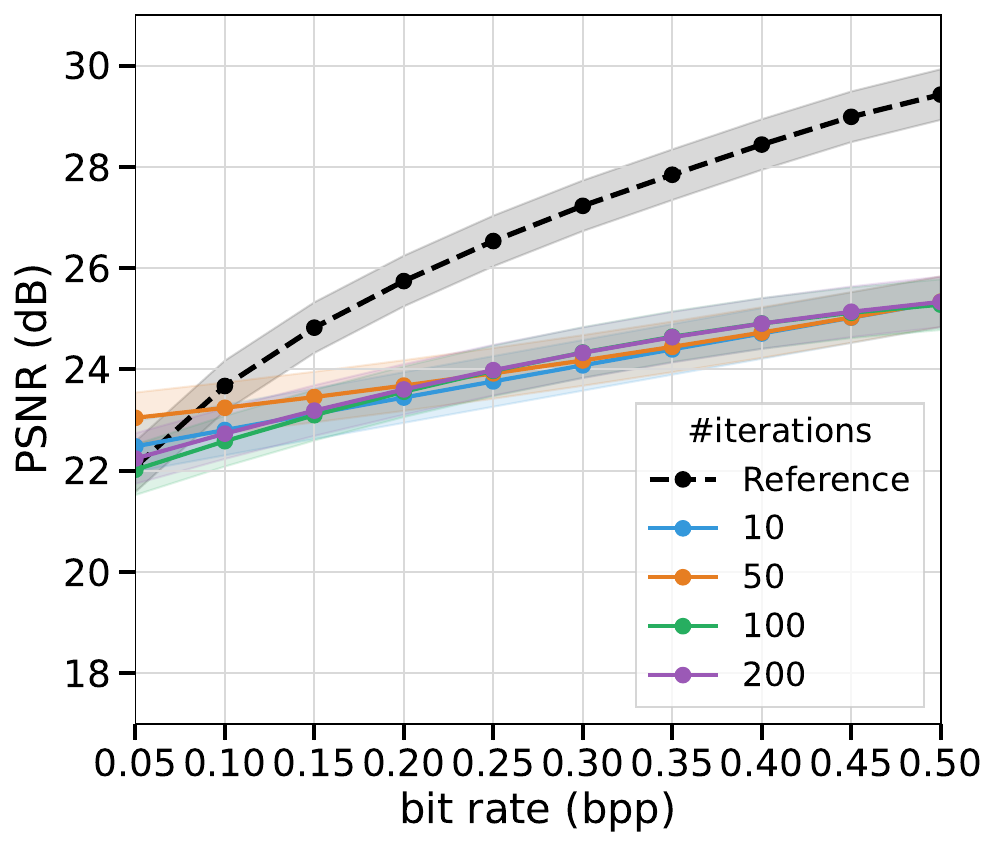}
      \includegraphics[width=0.5\columnwidth, height = 0.45\columnwidth]{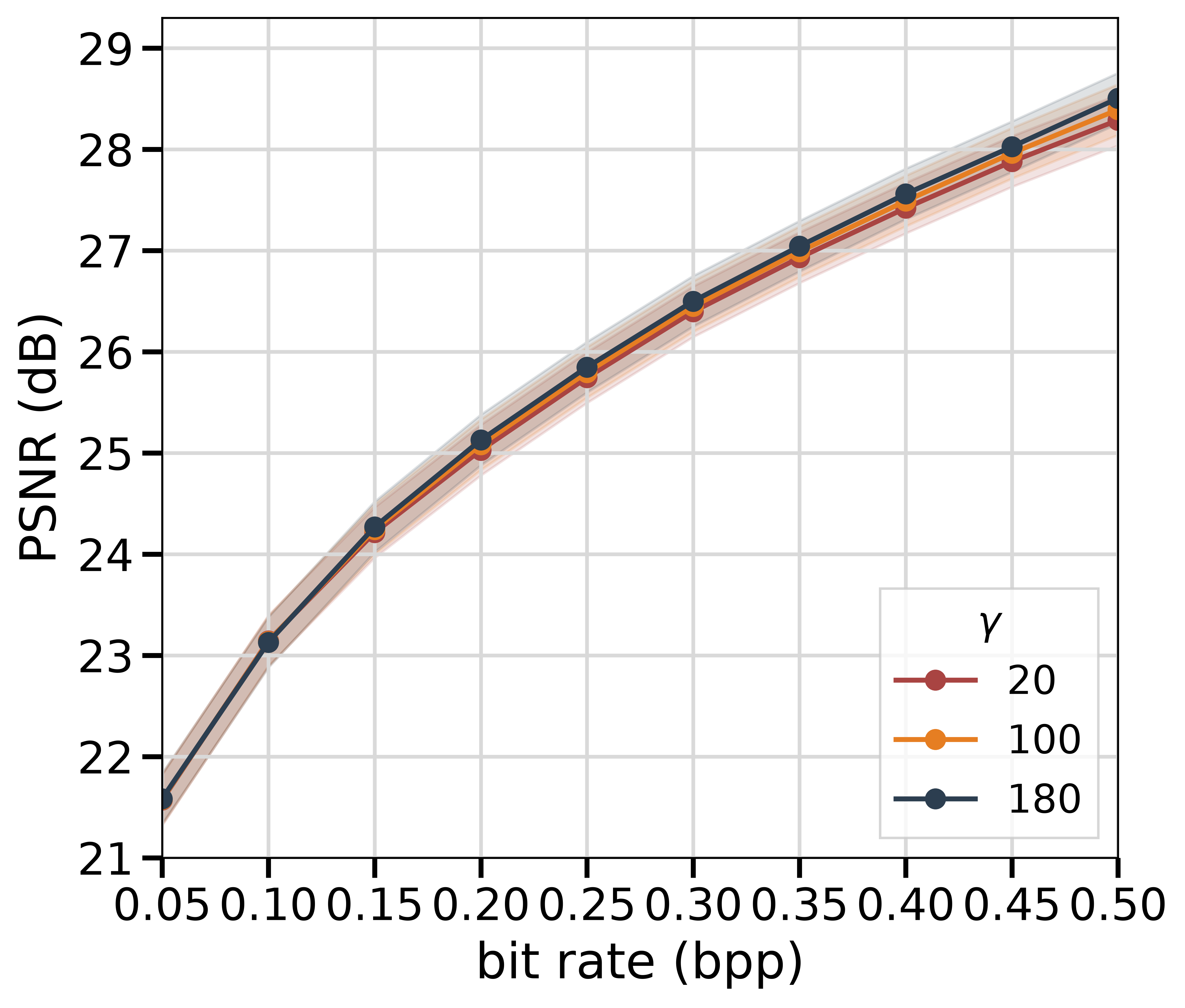}
    \caption{Sensitivity of  RQMF w.r.t. $[\alpha, \beta]$ bounds, number of iterations $K$, initialization of $(U,V)$ and regularization parameters $(M_u,M_v)$ on Kodak images.
    }
    \label{fig3}
\end{figure*}

\noindent In terms of quality Figure \ref{fig:transposed_comparison}  shows that RQMF maintains better visual fidelity than QMF, JPEG and SVD, effectively preserving color fidelity and intricate details, particularly in low-bit-rate scenarios. 
\begin{table}[h!]
\centering
\caption{Mean compression and decompression CPU times for the four methods at bit rates of 0.15 bpp and 0.25 bpp, measured on the Kodak and CLIC 2024 datasets.}
\begin{tabular}{l |c | c c c c}
\toprule
\textbf{Dataset} &
\textbf{Bit rate} &
\textbf{JPEG} &
\textbf{SVD} &
\textbf{QMF} &
\textbf{RQMF} \\
\midrule

\multirow{2}{*}{\textbf{Kodak}}
& 0.15 bpp & 5.4 ms & 1.6 ms & 3.2 ms & 3.2 ms \\
& 0.25 bpp & 5.2 ms & 1.5 ms & 3.1 ms & 3.1 ms \\

\midrule

\multirow{2}{*}{\textbf{CLIC 2024}}
& 0.15 bpp & 29.6 ms & 7.2 ms & 15.2 ms & 14.9 ms \\
& 0.25 bpp & 26.5 ms & 5.9 ms & 13.9 ms & 13.8 ms \\

\bottomrule
\end{tabular}
\end{table}

\noindent  The mean compression and decompression CPU times are reported  in Table 1. The compression and decompression speed of RQMF is comparable to QMF, benefiting from the fast matrix multiplication operations, which are significantly faster than the  DCT operations required by JPEG.


\subsection{Ablation studies}
\noindent We performed ablation studies to assess the influence of the factor bounds $[\alpha, \beta]$, initialization $(U_0,V_0)$, number of iterations ($K$) and regularization parameters $(M_u,M_v)$ on the compression for algorithm RQMF. The results are provided in  Figure  \ref{fig3}. In the left figure, the factor range ($\left[-16,15\right]$) provides the best PSNR–bit-rate trade-off for RQMF. Narrower factor bounds such as $\left[-8,7\right]$ reduce bit allocation and improve compression, but they are overly restrictive. The second figure illustrates the impact of the number of (full) iterations $K$ on the compression. Increasing the number of  iterations improves PSNR, although the gains eventually are insignificant (around $K=10$) and no longer justify the additional computational effort. Note that for $K=0$, RQMF employs no iterations and relies only on the SVD initialization, thus the poor result. In the third figure, we evaluate the impact of the initialization  on the performance of RQMF, executing  different number of iterations ranging from $10$ to $200$. It can be observed that initialization based on truncated SVD with $K=10$ (black line), consistently achieves higher compression quality than random initialization, even after 10 to 200 iterations.  Finally, from the right figure one can notice that our algorithm is robust w.r.t. a wide range of the regularization parameters $(M_u,M_v)$ (in simulations we choose $M_u = M_v$ as in \eqref{eq:M} and vary $\gamma$ for luminance and consider $\gamma+20$ for chrominance).


\subsection{Application of  image  compression on classification}
\noindent To assess the practical impact of our quantized compression scheme beyond rate-distortion metrics, we evaluate the compressed images in a downstream classification task. Our main motivation comes from embedded and edge devices, which typically rely on low-precision (integer) arithmetic due to limited memory, power and computational budgets. In such scenarios, the entire processing chain, from image preprocessing and storage to neural network inference, benefits from an integer representation of the data. Since the factor matrices produced by our RQMF are constrained to bounded integer values, the compressed images can be stored, transmitted and reconstructed using \texttt{INT8} arithmetic\footnote{The factor entries use a 5-bit integer alphabet \([-16,15]\), which can be stored using \texttt{INT8} data types.}, making our method naturally compatible with integer-only inference pipelines. The main benefits of quantization of the entire inference pipeline consist in a memory reduction of almost $4$x, $2$-$4$x speedup in computation and power savings proportional to memory savings (see Table II). \\

\noindent \textit{Network architectures:}  We consider three convolutional network (CNN) architectures representing distinct stages and design principles in the evolution of image classification networks. AlexNet serves as an early, parameter-heavy conventional CNN baseline; ResNet50 represents a deeper residual architecture; and MobileNetV2 represents a lightweight model explicitly designed for mobile and embedded inference. This selection enables us to investigate whether the effects of image compression and post-training quantization are consistent across architectures with substantially different capacities, computational structures and sensitivities to low-precision inference.


\begin{itemize}
\item \textit{AlexNet} \cite{alexnet}: one of the first deep convolutional networks to achieve breakthrough performance on ImageNet. It consists of 5 convolutional layers followed by 3 fully connected layers ($\approx 61$M parameters).

\item \textit{ResNet50} \cite{resnet}: a 50-layer deep residual network 
enabling the training of very deep models ($\approx 25.6$M parameters). 

\item \textit{MobileNetV2} \cite{mobilenetv2}: a lightweight architecture explicitly designed for mobile, embedded and edge devices ($\approx 3.5$M parameters). It replaces standard convolutions with depthwise separable convolutions, 
drastically reducing the number of operations and parameters. 
Hence, MobileNetV2 is the most relevant architecture for  embedded scenarios.
\end{itemize}

\noindent \textit{Evaluation pipeline:}  We consider  ImageNet validation set ($50.000$ images), where compression  is applied to the full-resolution original image, and  then we use the standard evaluation preprocessing steps: resize to $256 \times 256$, center crop to $224 \times 224$,  scale to $[0, 1]$ and finally normalized. For inference we consider two numerical representations of each network: (i) the original pre-trained model using $32$-bit floating point (FP32) arithmetic, and (ii) its post-training quantized (PTQ) counterpart, in which both weights and activations are quantized to $8$-bit integers \cite{gholami2022survey, wang2022niti}. We  run these convolutional networks on RQMF, SVD and JPEG compressed images at matched bit rates. The classification performance is measured via the standard Top-1 and Top-5 accuracies.\\

\noindent \textit{Post-training quantization of the networks:} In order to decrease the storage and compute requirements of the model during inference its parameters are  stored as integers with a low number of bits. To obtain the integer counterparts of the  networks, we employ post-training quantization (PTQ), which is a family of techniques that convert a pre-trained floating point network to low-precision (integer) arithmetic \emph{without any retraining} \cite{jacob2018quantization, krishnamoorthi2018quantizing}. Compared to quantization-aware training, PTQ is simpler to deploy, does not require   access to the original training pipeline, making it the method of choice for embedded and edge scenarios. For the reader’s convenience, in what follows we briefly outline the principles and procedures underlying post-training quantization of neural networks based on the works \cite{jacob2018quantization, krishnamoorthi2018quantizing}.
In \emph{weight and activation quantization}, the network weights and the activations are converted to low precision (e.g., \texttt{INT8}), which enables integer-only inference. Our presentation of  PTQ below follows  \cite{krishnamoorthi2018quantizing}.

\medskip
\noindent\textit{Uniform affine quantization.}
Let $x \in [x_{\min}, x_{\max}]$ be a real-valued quantity (a weight or
an activation) to be mapped onto $b$-bit integers with
$N = 2^{b}$ representable levels ($N = 256$ for \texttt{INT8}). The
\emph{uniform affine quantizer} is parameterized by a scale
$\Delta > 0$ and an integer zero-point $z$:
\begin{equation}
\Delta \;=\; \frac{x_{\max} - x_{\min}}{N - 1},
\qquad
z \;=\; \operatorname{round}\!\left(-\frac{x_{\min}}{\Delta}\right),
\label{eq:scale-zp}
\end{equation}
and maps a real value $x$ to its integer representation via
\begin{equation}
x_{\mathrm{int}}
\;=\;
\operatorname{clamp}\Bigl(0,\, N-1,\;
\operatorname{round}\bigl(\tfrac{x}{\Delta}\bigr) + z\Bigr).
\label{eq:quant}
\end{equation}
This zero-point $z$ guarantees that the real value $0$ is represented \emph{exactly}, an important property for common operations such as zero padding; for one-sided
distributions the range $[x_{\min}, x_{\max}]$ is relaxed to include
zero. A frequently used simplification is the \emph{symmetric} quantizer,
which fixes $z = 0$ and uses a range symmetric around zero:
\begin{align*}
&\Delta \;=\; \frac{\max\bigl(|x_{\min}|, |x_{\max}|\bigr)}{N/2 - 1}\\
&x_{\mathrm{int}} \;=\;
\operatorname{clamp}\Bigl(-\tfrac{N}{2},\, \tfrac{N}{2}-1,\;
\operatorname{round}\bigl(\tfrac{x}{\Delta}\bigr)\Bigr),
\label{eq:symmetric}
\end{align*}
which removes the cross-terms involving the zero-point in  integer
convolution and thereby simplifies fast Single Instruction Multiple Data implementations at the cost of a slightly less efficient use of  integer grid for skewed distributions.

\medskip
\noindent\textit{Granularity of quantization.}
A key design choice in  PTQ is the \emph{granularity} of the quantizer. With
{per-layer} (per-tensor) quantization, a single pair
$(\Delta, z)$ is used for an entire weight tensor
$W \in \mathbb{R}^{K \times K \times C_{\mathrm{in}} \times C_{\mathrm{out}}}$,
whereas {per-channel} quantization assigns an individual scale to
each convolutional kernel, i.e., in  symmetric case to each output channel
$c \in \{1, \dots, C_{\mathrm{out}}\}$:
\begin{equation}
\Delta_{c} \;=\;
\frac{\max\limits_{k,l,m}\bigl|W(k,l,m;c)\bigr|}{N/2 - 1},
\label{eq:per-channel}
\end{equation}
Since  scale is constant within each kernel, both granularities remain compatible with efficient integer dot-product and convolution kernels; activations, in contrast, are quantized per-tensor, as per-channel activation scales  complicate  inner accumulation of {convolutions} and {matrix multiplication}~operations.

\medskip
\noindent\textit{Range estimation.}
For weights, the quantizer parameters are obtained directly from the
empirical extrema, $x_{\min} = \min_i w_i$ and $x_{\max} = \max_i w_i$
(per tensor or per channel). Since activation ranges are input-dependent, they must be estimated in a short {calibration} step, in which a small set of representative images (typically around one hundred mini-batches) is propagated through the network and the dynamic ranges of each layer are collected   as exponential moving averages of the per-batch  extrema:
\begin{align}
&x_{\min}^{(t)} = (1-\alpha)\, x_{\min}^{(t-1)} + \alpha \min(x^{(t)}), \nonumber \\
&x_{\max}^{(t)} = (1-\alpha)\, x_{\max}^{(t-1)} + \alpha \max(x^{(t)}),
\label{eq:ema}
\end{align}
with smoothing factor $\alpha \in (0,1)$; more elaborate criteria, such
as selecting the clipping range that minimizes the Kullback--Leibler
divergence between the original and the quantized distribution, can
further improve accuracy \cite{krishnamoorthi2018quantizing}. 

\medskip
\noindent As shown in \cite{krishnamoorthi2018quantizing}, per-channel
quantization of weights combined with per-layer quantization of
activations achieves classification accuracies within $2\%$ of the
floating point baseline for a wide range of architectures at $8$ bits. 
Larger, over-parameterized networks (e.g., ResNets) are generally more robust
to quantization than compact ones (e.g., MobileNet). In this work we adopt the following PTQ recipe: symmetric per-channel quantization of weights (see eq.~\eqref{eq:per-channel}), asymmetric per-tensor quantization of
activations (see eqs.~\eqref{eq:scale-zp}--\eqref{eq:quant}) and
calibration-based range estimation (see eq.~\eqref{eq:ema}) - to obtain the
\texttt{INT8} counterparts of the networks used in our classification
experiments.

\begin{figure}
    \centering
    \includegraphics[scale=0.25]{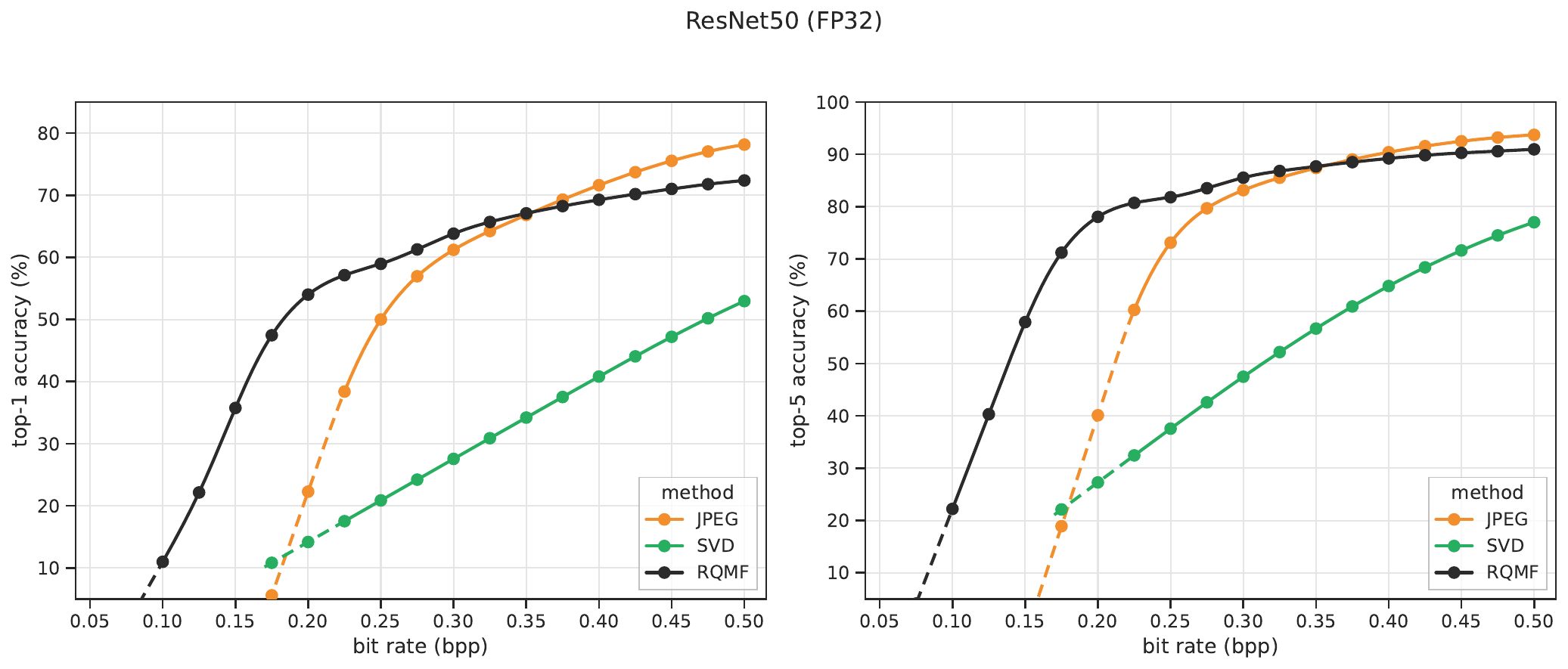}
    \includegraphics[scale=0.25]{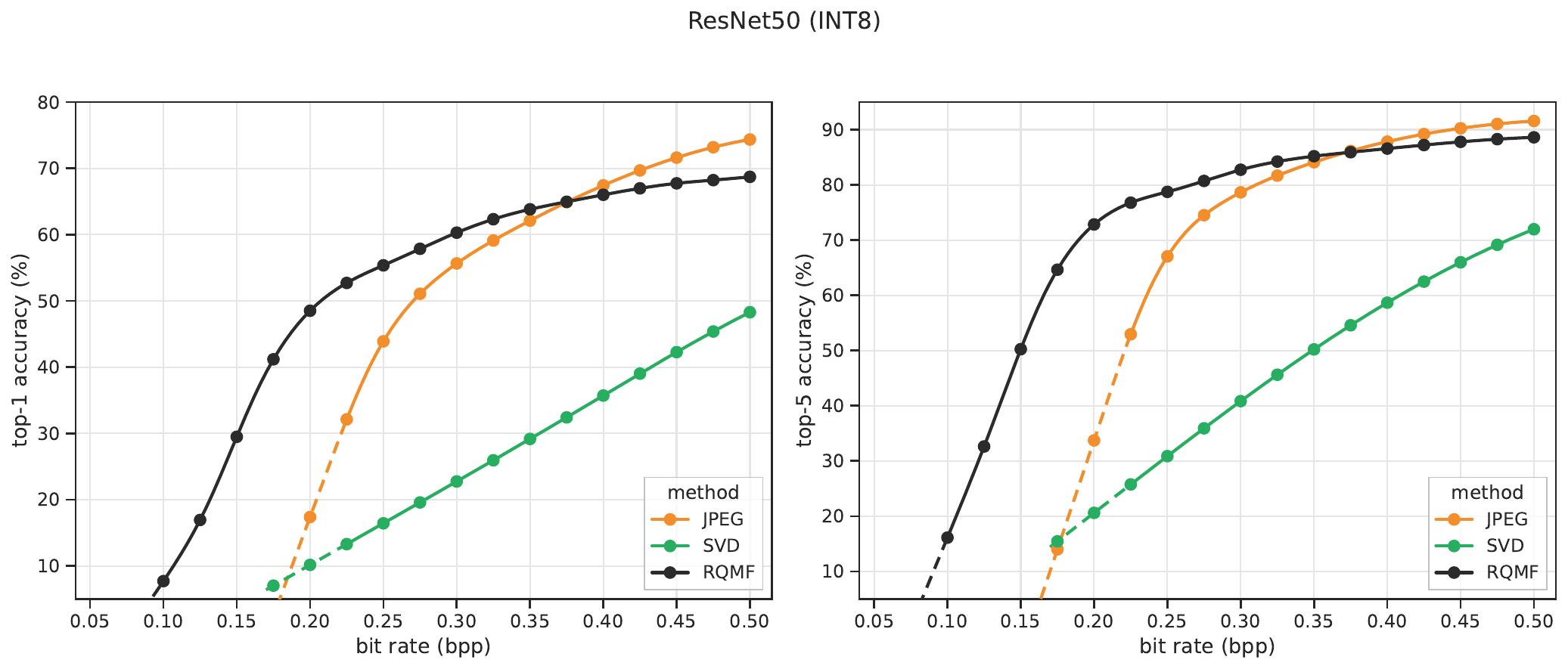}
    \includegraphics[scale=0.25]{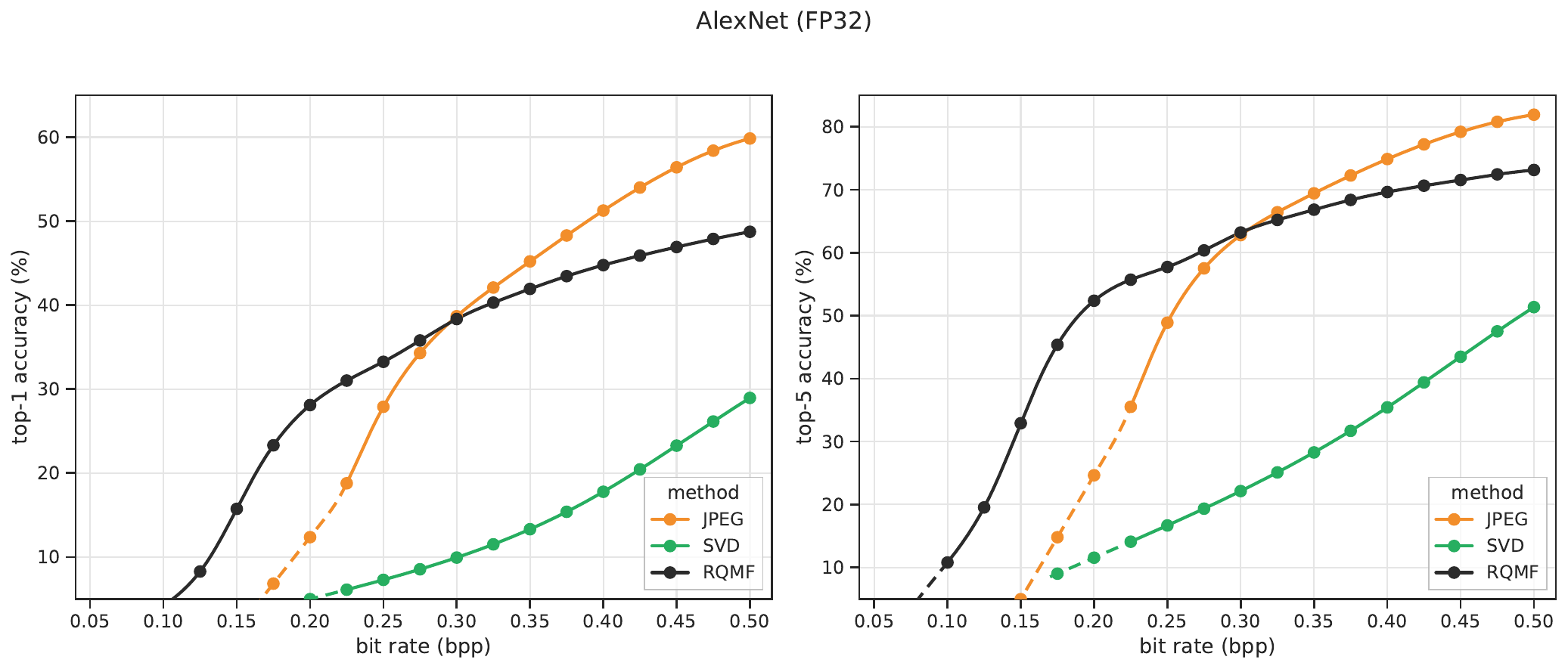}
    \includegraphics[scale=0.25]{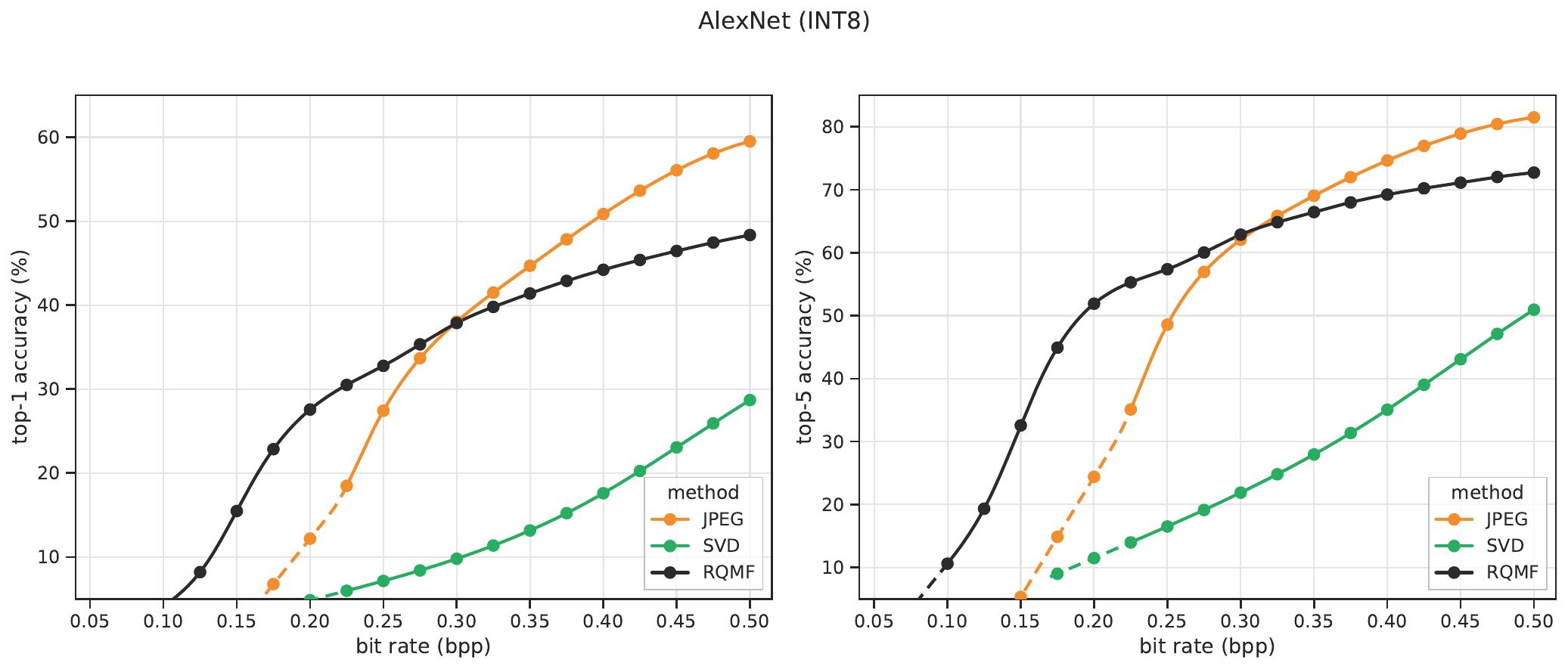}
    \includegraphics[scale=0.25]{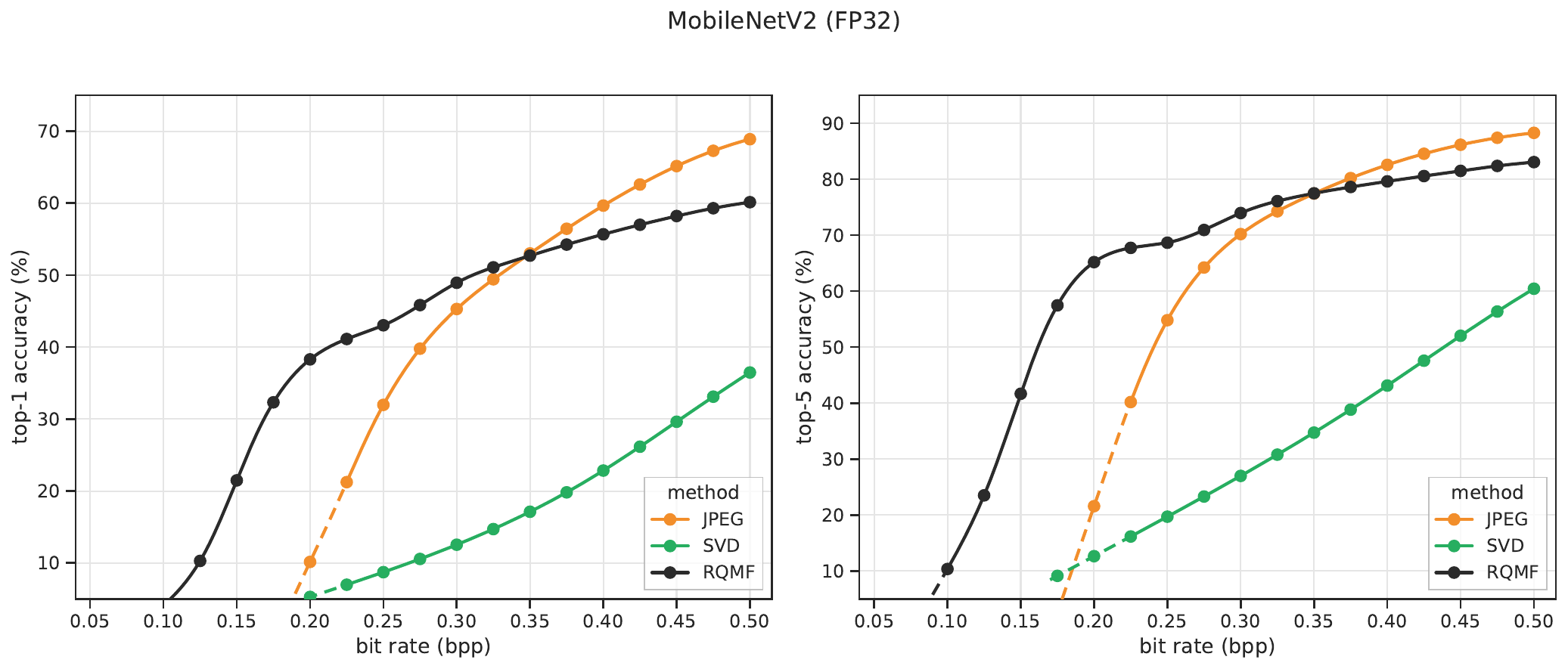}
    \includegraphics[scale=0.25]{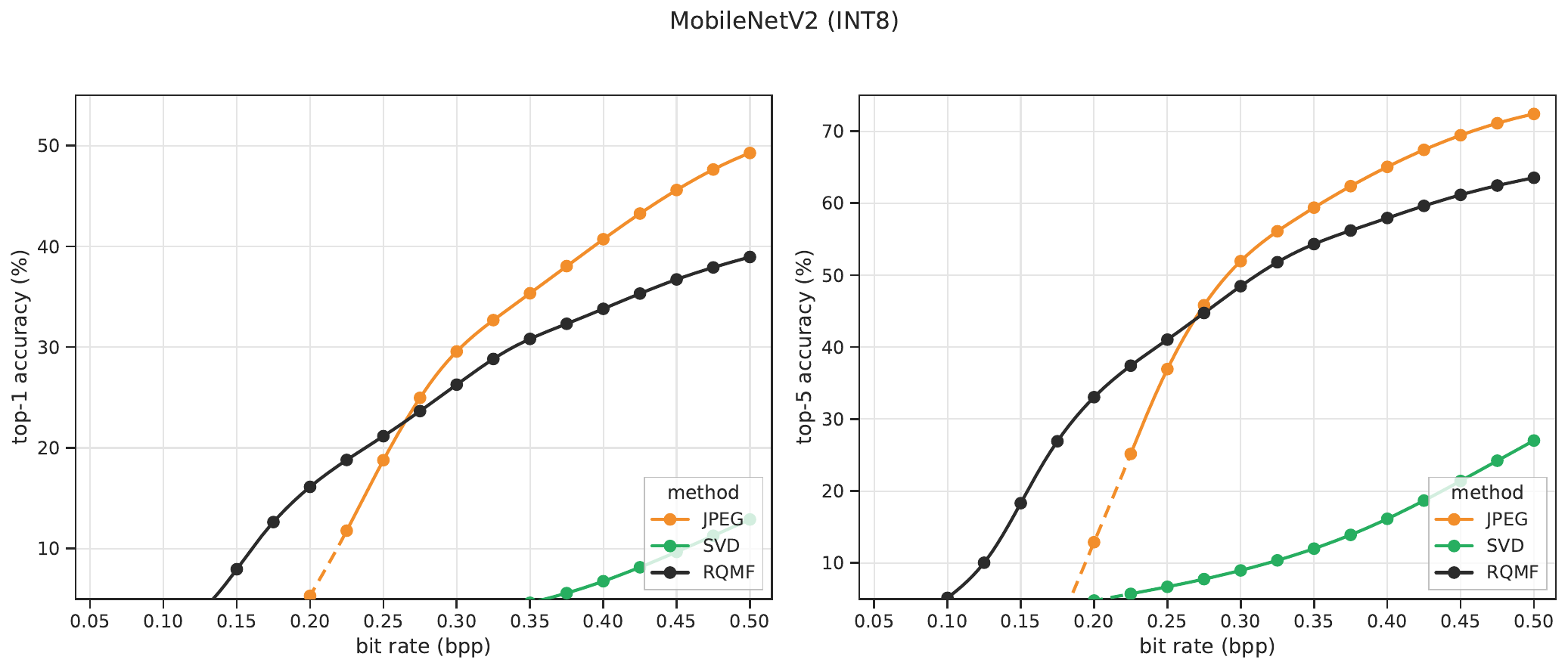}
    \caption{Top-1 (left) and top-5 (right) accuracies  on the ImageNet validation set at several bit rates for ResNet50, AlexNet and MobileNetV2 in both FP32 and INT8 representations.}
    \label{classification}
\end{figure}

\medskip

\noindent \textit{Numerical results for classification}.
Figure \ref{classification} and Table \ref{tab:results_025bpp} 
report Top-1 and Top-5 accuracies of the three networks (AlexNet,
MobileNetV2 and ResNet50), in both \texttt{FP32} and \texttt{INT8}
representations, on the original ImageNet dataset and also compressed with RQMF, JPEG and SVD at different bit rates. First, comparing the two numerical representations,
post-training quantization is virtually lossless for AlexNet and ResNet50
(at most a $4$-point Top-1 drop), confirming that the quantization error is
negligible at the classification level for these architectures. MobileNetV2
is more sensitive to quantization, as expected for compact models built on
depthwise separable convolutions, which are notoriously challenging for
standard post-training quantization; notably, this gap is already present on
uncompressed images and is therefore independent of the image codec. Second,
the images compressed with RQMF preserve the semantic content required for
classification: RQMF yields higher accuracies than JPEG at low bit rates
(below ${\sim}0.3$~bpp) for all networks and both precisions, consistent
with the rate-distortion results of the previous section. Table \ref{tab:results_025bpp} showcases the performance of the three networks at low bit rate (0.3 bpp) compression on ImageNet. One can notice that the average inference time per image (including memory access, transfer times and preprocessing  in order to simulate a real embedded system scenario) has a 2-3x speedup for \texttt{INT8} representation of the networks.  Moreover, one can observe approximately 4x reduction in memory usage. Overall, these results indicate that the proposed integer-based compression (algorithm RQMF)  can be combined with integer-only network inference into a fully quantized end-to-end pipeline suitable for embedded devices, with negligible impact on the
classification performance.

\begin{table}[h]
\centering
\caption{Top-1 and top-5 accuracies and average inference time on the ImageNet validation set at a target bit rate of approximately $0.30$~bpp. }
\label{tab:results_025bpp}
\begin{tabular}{l|l|l|l|p{0.5cm}p{0.5cm}p{0.7cm}p{0.7cm}}
\toprule
CNNs & Param. & Prec. & Meth. & \shortstack{Top-1\\ (\%)} & \shortstack{Top-5 \\(\%)} & \shortstack{Time \\(ms/img)} & \shortstack{Disk \\(MB)} \\
\midrule
\multirow{6}{*}{\rotatebox{90}{ResNet50}} & \multirow{6}{*}{25.6M} & \multirow{4}{*}{FP32} & Orig. & 80.34	& 95.13 & 27.61 & \multirow{4}{*}{97.79} \\
&  &  & JPEG & 59.57 & 81.85 & 14.70 &\\
&  &  & SVD & 14.20 & 27.33 & 21.01 &  \\
&  &  & RQMF & 64.15 & 85.80 & 22.34 &  \\
\cmidrule(lr){3-8}
&  & \multirow{4}{*}{INT8} & Orig. & 77.36	&93.63
 & 13.87 & \multirow{4}{*}{24.95} \\
&  &  & JPEG & 53.85 & 77.05 & 10.45 \\
&  &  & SVD & 10.17 & 20.64 & 10.19 &  \\
&  &  & RQMF & 60.63 & 83.01 & 10.55 &  \\
\midrule
\multirow{6}{*}{\rotatebox{90}{MobileNetV2}} & \multirow{6}{*}{3.5M} & \multirow{4}{*}{FP32} & Orig. & 72.01 &90.62 & 27.86 & \multirow{4}{*}{13.60} \\
&  &  & JPEG & 43.14 & 67.92 & 14.39 & \\
&  &  & SVD & 5.30 & 12.68 & 21.17 &  \\
&  &  & RQMF & 49.34 & 74.36 & 22.53 &  \\
\cmidrule(lr){3-8}
&  & \multirow{4}{*}{INT8} & Orig. & 52.79&	76.39
 & 10.08 & \multirow{3}{*}{3.75} \\
&  &  & JPEG & 27.81 & 49.61 & 6.41 & \\
&  &  & SVD & 1.75 & 4.79 & 8.22 &  \\
&  &  & RQMF & 26.68 & 49.01 & 8.60 &  \\
\midrule
\multirow{6}{*}{\rotatebox{90}{AlexNet}} & \multirow{6}{*}{61.1M} & \multirow{4}{*}{FP32}& Orig. & 56.51 &	79.07
 & 28.53 & \multirow{4}{*}{233.09} \\
&  &  & JPEG & 36.95 & 60.78 & 13.73 & \\
&  &  & SVD & 4.99 & 11.58 & 22.50 &  \\
&  &  & RQMF & 38.69 & 63.57 & 23.97 &  \\
\cmidrule(lr){3-8}
&  & \multirow{4}{*}{INT8} & Orig. & 56.16	&78.87 & 9.22 & \multirow{4}{*}{58.47} \\
&  &  & JPEG & 36.30 & 60.07 & 5.29 &  \\
&  &  & SVD & 4.87 & 11.51 & 7.26 & \\
&  &  & RQMF & 38.21 & 63.23 & 7.69 &  \\
\bottomrule
\end{tabular}
\end{table}


\section{Conclusions}
In this paper we have combined cyclic coordinate minimization with quadratic regularization that leads to a new algorithm for  solving  general (possibly nonconvex) composite  optimization problems.  We have performed  a full  convergence  analysis  of  our  method deriving rates ranging from sublinear to superlinear, depending on the assumptions on the problem. Then, we have specialized our algorithm to  quantized matrix factorization problems that arise in e.g.,   image compression. Finally, we have evaluated both  the original and  quantized compressed images in a classification task using several  well-known convolutional networks. The numerical results on image compression and classification using real data have shown the flexibility and  efficiency of our  algorithm when compared to well-established methods from the literature.

\bibliographystyle{abbrv}
\bibliography{ifacconf}            
\end{document}